\documentclass[a4paper,12pt]{article}

\usepackage[T2A]{fontenc}
\usepackage[utf8]{inputenc}
\usepackage[russian, english]{babel}

\usepackage[a4paper,top=1cm,bottom=2cm,left=1.5cm,right=1.5cm,marginparwidth=1.75cm]{geometry}

\usepackage{amsmath}
\usepackage{amssymb}

\usepackage{graphicx}
\graphicspath{{./images/}}

\usepackage{hyperref}
\usepackage{array}
\usepackage{csquotes}
\usepackage[shortlabels]{enumitem}
\usepackage[backend=biber,
            sorting=nyt,
            style=nature,
            maxbibnames=2,
            autolang=other,
            ]{biblatex}

\usepackage{amsthm}

\newtheorem{theorem}{Theorem}[section]
\newtheorem{corollary}{Corollary}[theorem]
\newtheorem{lemma}[theorem]{Lemma}
\newtheorem{conjecture}[theorem]{Conjecture}

\theoremstyle{definition}
\newtheorem*{definition}{Definition}

\def\ZZ{{\mathbb Z}}

\def\NN{{\mathbb N}}

\def\HG{{\mathcal G}}
\def\gplus#1{\mathop{+\!\!_{_#1}}}
\def\gtimes#1{\mathop{\times\!\!_{_#1}}}
\def\hats{{\sc Hats\ }}
\def\hgn{\text{\rm HG}}

\usepackage{tikz}
\usetikzlibrary{
  arrows,
  mindmap,
  trees,
  shadows,
  graphs,
  graphs.standard,
  decorations.markings,
  intersections,
  calc
}

\tikzstyle{win_nodes}=[circle, draw=black, inner sep=1.5, fill=black]
\tikzstyle{win_edges}=[draw=black, thick]
\tikzstyle{lose_nodes}=[circle, draw=black, inner sep=1.5, fill=black]
\tikzstyle{lose_edges}=[draw=black, thick]

\definecolor{lgreen}{RGB}{180, 180, 180}
\tikzstyle{vecArrow} = [thick, decoration={markings,mark=at position
   1 with {\arrow[semithick]{open triangle 60}}},
   double distance=1.4pt, shorten >= 5.5pt,
   preaction = {decorate},
   postaction = {draw,line width=1.4pt, white,shorten >= 4.5pt}]
\tikzstyle{innerWhite} = [semithick, white,line width=1.4pt, shorten >= 4.5pt]

\title{Cliques and constructors in ``Hats'' game}
\author{Konstantin Kokhas\thanks{St.Petersburg State University, St.Petersburg,
    Russia. Email: kpk@arbital.ru} \and Aleksei Latyshev\thanks{ITMO University,
    St.Petersburg, Russia. Email: aleksei.s.latyshev@gmail.com} \and Vadim
  Retinsky\thanks{National Research University Higher School of Economics,
    Moscow, Russia. Email: viretinskiy@gmail.com}}

\font\shmt=chess10
\font\shmtt=chess20

\begin{document}
\maketitle
\begin{abstract}
  The following general variant of deterministic Hats game is analyzed.
  Several sages wearing colored hats occupy the vertices of a graph, the $k$-th sage
  can have  hats of one of $h(k)$ colors. Each sage tries to guess the color of
  his own hat merely on the basis of observing the hats of his neighbors
  without exchanging any information. A predetermined guessing strategy is
  winning if it guarantees at least one correct individual guess for every
  assignment of colors.

  For complete graphs and for cycles we solve the problem of describing functions
  $h(k)$ for which the sages win. We demonstrate here winning strategies for the
  sages on complete graphs, and analyze the Hats game on almost complete graphs.
  We develop ``theory of constructors'', that is
  a collection of theorems demonstrating how one can construct new graphs for
  which the sages win. We define also new game ``Check by rook'' which is
  equivalent to Hats game on 4-cycle and give complete analysis of this game.
\end{abstract}

\section{Introduction}

The Hats game is an interesting mathematical puzzle attracting attention of
many mathematicians for many years. In the classical version of the problem
there is a set of $n \geq 2$ players (sages) and adversary, who puts a hat of
one of $n$ colors to the head of each sage. Each sage sees the hats of the other
sages, but does not see his own hat. Taking into account this information only, he tries to guess the color of hat he is wearing. The goal of the sages is to guarantee that at least one of 
them guesses the color of hat correctly whatever the hats arrangement is. The
players are allowed to discuss and fix a strategy before the hat assignment.
After that any communication is prohibited. When the sages simultaneously say
their guesses, winning condition would be checked (is it true, that at least one
of the sages guesses correctly). The problem is  ``whether the sages can
guarantee a win?''

The answer to the above problem is ``Yes!''. It can be justified gracefully.
Let us enumerate the sages and identify the colors of hats with residues modulo
$n$. Every sage sees every hat except his own. Let the ``i''-th sage checks the
hypothesis that the sum of all colors, including his own, equals $i$ modulo $n$
and say the corresponding remainder. It's clear that the hypothesis of exactly
one of the sage is true, regardless of the hats arrangement. Thus this
sage guesses correctly the color of his hat.

A natural generalization of this problem is a game in which every sage can see
only some part of the others. Formally, let the sages be located in vertices
of some graph (``visibility graph''), the sage $i$ can see a color of the hat of
the sage $j$ if and only if there is an edge $(i, j)$ in the graph. This generalization was
introduced in~\cite{butler2009hat} and further was studied in a number of
papers~\cite{Gadouleau2015, Gadouleau2018, cycle_hats}. For example, the
connection of this Hats game with dynamical systems and coding theory was
analyzed in~\cite{Gadouleau2018}. M.\,Farnik define $\hgn(G)$ in his PhD thesis~\cite{Farnik2015} as the maximal number of hat colors, for which the sages can guarantee 
a win. He got some estimations of $\hgn(G)$ in terms of maximum
degree of the graph and graph chromatic number. In~\cite{Alon2018} N.\,Alon et al.\ studied
$\hgn(G)$ for some classes of graphs using mostly probabilistic methods.
The connection between $\hgn(G)$ and other graph parameters was considered by
Bosek et al.~\cite{bosek19_hat_chrom_number_graph}.

W.\,Szczechla~\cite{cycle_hats} got complicated result that in case of three
colors. The sages can win on cycles with $n$ vertices if and only if $n$ is divisible by $3$ or $n
= 4$. Complete list of graphs on which sages win in case of three colors can be found
in~\cite{Kokhas2018}.

In addition to the above, a lot of other variants of the Hats game were
considered. For example, M.\,Krzywkowski's described in~\cite{krzywkowski2010variations}  36 variants of the game rules and
most of them are  probabilistic.  Description of important results and
applications of this game can be found in the same paper.

In paper~\cite{Kokhas2018} authors explain how to reduce the problem of finding
winning strategies for sages on graphs (generally speaking, this problem is very cumbersome)
to SAT (the Boolean satisfiability problem). This makes it possible to study winning strateies by computer for small quite efficiently.

In the present paper we consider a modification of the classical deterministic game on
a graph, in which the sages have different number of possible hat colors.
This modification is not only of its own interest but allows one to find more simple
strategies in the classical game, where the number of colors is constant.

This text is combined from papers \cite{KL19} and \cite{KLR19}.

An application of the technique developed in the present paper is given in~\cite{KL20}, where we  build a planar graph with hat guessing number at least 14.

We introduce the following notation.

$\bullet$ $G = \langle {V, E} \rangle$ is a visibility graph, i.\,e.\ graph, at the
vertices of which the sages are located. We often identify the sages and the
vertices of $G$.

$\bullet$ $h\colon V\to \NN$ is a hat function, or ``hatness'' for short,
$h(v)$ is a number of possible colors for the hat of sage $v$. For sage $A\in V$ let \emph{hatness of sage
A} be the value of $h(A)$. We assume that the list of colors using in this game is
known in advance,  and the color of the hat of $A$ is one of the first $h(A)$ colors in this list.
We often identify the set of
possible hat colors of sage $A$ with set of residues modulo $h(A)$.

\begin{definition}
  The \hats game is the pair $\HG = \langle {G, h} \rangle$, where $G$ is
  a visibility graph, and $h$ is a hat function. The sages are located at
  the vertices of visibility graph $G$ and participate in a \emph{test}. Before the test the sages should determine a public deterministic strategy. During the
  test, every sage $v$ gets a hat of one of $h(v)$ colors. The sages try to guess
  color of their own hats according the chosen strategy and if for each hats arrangement at least one of them guesses correctly, we say
  that the sages \emph{win} or the game is winning. We call the graph in this case
  also winning, keeping in mind that this property depends also on the
  hat function. The game in which the sages have no winning strategy is said to be
  \emph{losing}.

  A game $\HG_1=\langle {G_1, h_1} \rangle$ is \emph{a subgame} of the game $\HG =
  \langle {G, h} \rangle$, if $G_1$ is a subgraph of the graph $G$ and
  $h_1=h\Big\vert_{V(G_1)}$.

  When the adversary puts hats on the heads of all sages, i.\,e.\ assigns a possible
  hat color to every sage, we obtain \emph{hats arrangement}. Formally, every hats
  arrangement is a function $\varphi\colon V(G)\to \ZZ$, where
  $0\leq\varphi(v)\leq h(v)-1$ for all $v\in V(G)$.
\end{definition}

We use standard notations of graph theory:
$C_n$ is an $n$-vertex cycle graph, $P_n$ is an $n$-vertex path,
$P_n(AB)$ is an $n$-vertex path with ends $A$ and $B$, $K_n$ is a
complete graph with $n$ vertices, $N(v)$ or $N_G(v)$ is a set of neighbors of
vertex $v$ in graph $G$.

Denote by $G^A$ the graph, in which one of vertices is $A$. This notation
is used emphasize that graphs under consideration share
common vertex $A$.

By $\langle {G, k} \rangle$ we denote the game on graph $G$ with constant hat function that is equal to $k$. For example, in these terms, the classical game described in the first paragraph is $\langle {K_n, n} \rangle$.

A \emph{strategy} of the sage in vertex $A$ is a function $f_A$ that puts into correspondence to each hats arrangement on $N(A)$  possible color of sage $A$'s hat (i.e. an integer from 0 to $h(A)-1$).  \emph{Collective strategy of the sages} is just the set $\{f_A\mid A\in V(G)\}$.

\medbreak

In the second section we consider the \hats game on complete and ``almost complete'' graphs. The main result here is theorem~\ref{thm:clique-win}.

In the third section we develop ``theory of constructors'', which is a set of theorems that
allow one to construct new winning graphs from existing ones.

In the fourth section we develop new elegant approach to the \hats
game. We describe new game ``Rook check'' that is, in fact, equivalent to the \hats game on a
$4$-cycle. It expands the arsenal of combinatorial tools for constructing
strategies. We present a complete research of game ``Rook check''
and discuss some of its variations.

In the fifth section we analyze the \hats game on cycles with
arbitrary hat functions.

\section{\hats game on complete and almost complete graphs}

\subsection{Game on complete graphs}

In this section we describe the game on a complete graph with vertices $A_1$,
$A_2$, \dots, $A_n$ and arbitrary hat function $h$. Let $a_i=h(A_i)$. The following theorem completely solves
the problem ``for which hat functions on a complete graph the sages win?''

\begin{theorem}\label{thm:clique-win}
  Let hatnesses of $n$ sages located in vertices of complete graph, be $a_1$,
  $a_2$, \dots, $a_n$. Then the sages win if and only iff
  \begin{equation}
    \frac1{a_1}+\frac1{a_2}+\ldots+\frac1{a_n}\geq 1.
    \label{eq:clique-win}
  \end{equation}
\end{theorem}

\begin{proof}
  The necessity of condition~\eqref{eq:clique-win} is obvious: for each strategy of the sages, the $i$-th sage
  guesses correctly exactly on $\frac 1{a_i}$ of all hat arrangements, so if the sum is
  less than 1, there exists an arrangement for which no one guesses correctly.

  We give two proofs of the sufficiency of condition~\eqref{eq:clique-win}. The
  first one uses Hall's marriage theorem and the second one presents the
  strategy, that generalizes the arithmetic strategy for the classical game.

  \medskip
  \emph{Proof 1.}
  Let us prove that if the sum is greater than or equal to 1, then the sages win. The
  existence of a winning strategy is proved by using Hall's marriage theorem.

  For each sage $i$, we split the set of all hat arrangements into subsets of $a_i$
  elements each in the following way. Delete the color $c_i$ of the $i$-th sage from
 each hat arrangement. For the remaining set $c=(c_1, \dots, c_{i-1},
  \bar{c_i}, c_{i+1}, \dots, c_n)$ (symbol ``bar''  means that this color is
  omitted) put
  $$
  A_c^i=\{(c_1,\dots, c_{i-1}, \ell, c_{i+1}, \dots, c_n ) \mid 0\leq \ell\leq
  a_i-1
  \}.
  $$
  Set $A_c^i$ consists of  ``potentially possible'' hat arrangements from the
  point of view of the $i$-th sage:  he sees that colors of the other sages form
  a set $c$ and mentally appends to it all possible colors $\ell$ of his own hat.
  Bearing in mind the application of Hall's theorem, we associate the sets $A_c^i$ with ``girls'' and hat arrangements with ``boys''.
  The boy $s$ and girl $A_c^i$ \emph{know each other} if the hat arrangement $s$ is an element of $A_c^i$. Every
  boy knows $n$ girls, and for each $i$ every man knows exactly one girl of
  type $A_c^i$. Every girl~$A_c^i$ knows exactly $a_i$ boys.

 Let us prove that there exists a matching sending each boy to a girl. It 
 suffices to check the theorem condition, i.e., that every $m$ boys know together at
  least $m$ girls. Consider an arbitrary set of $m$ boys. Since for each $i$,
 the girl $A_c^i$ knows exactly $a_i$ boys, any $m$ boys know in
  total at least $\frac m{a_i}$ girls of type $A_c^i$ for each $i$. 
  Summing over $i$, we find that the
  total number of girls familiar with these $m$ boys is at least
  $$
  \frac m{a_1}+\frac m{a_2}+\ldots+\frac m{a_n}\geq m.
  $$
  This shows that the condition f the Hall's theorem is satisfied.

  Thus, there exists a matching that assigns to each hat
  arrangement a set of type $A_c^i$. Note that if the equality
  $\frac1{a_1}+\frac1{a_2}+\ldots+\frac1{a_n}= 1$ holds, then this matching selects, in fact, one element in each set $A_c^i$. Otherwise, if 
  $\frac1{a_1}+\frac1{a_2}+\ldots+\frac1{a_n}> 1$, then ``there are
  lonely girls'', i.\,e.\ no elements are selected in some sets $A_c^i$.

  The constructed matching allows to define a strategy for the sages. Let the $j$-th sage
  act by the rule: looking at hats of the other sages, i.\,e. at the set of colors
  $$
  c =(c_1,\dots, c_{j-1}, c_{j+1}, \dots, c_n),
  $$
  he reconstructs  the set $A_c^j$ which, in fact, consists of all possible ways
  to supplement the set $c$ to the hat arrangement on the whole graph. The sage
  should say the color marked in set $A_c^j$ by our matching (if there is
  no marked element, he says color arbitrarily). Since each hat arrangement is
  mapped by our matching to the selected element of one of sets~$A_c^i$,  
  the $i$-th sage guesses correctly his own color for this hat arrangement.

  \medskip
  \emph{Proof 2.}
  Let $N=\text{LCM}(a_1, a_2, \dots, a_n)$ (the least common multiple). For $k$ from
  $1$ to $n$ set $d_k=N/a_k$. We identify the set of all possible hat colors of
  the $k$-th sage and the set  of integers $\{d_k, 2d_k, \dots, a_kd_k\}$
  modulo~$N$. Now we describe the winning strategy of the sages. Let the $k$-th sage
  get hat of color $x_kd_k$, where $x_k\in\{1,2,\dots,a_k\}$ ($1\leq k\leq n$).
  Let
  $$
  S=x_1d_1+x_2d_2+\ldots +x_nd_n\pmod{N}.
  $$
  Each sage, seeing those around him, knows all the summands of this sum, except for his own. Making assumption about the value of the sum, he can calculate the color
  of his own hat. Let the first sage check the hypothesis $S\in\{1, 2, \dots, d_1\}$;
  the second sage check the hypothesis $S\in\{d_1+1, d_1+2, \dots, d_1+d_2\}$ and so
  on, the $n$-th sage check the hypothesis $S\in\{d_1+d_2+\ldots+d_{n-1}+1, \dots,
  d_1+d_2+\ldots+d_{n-1}+d_n\}$. The hypothesis of the \hbox{$k$-th} sage involves
  $d_k$ consecutive integers, among which exactly one is divisible by~$d_k$.
  This integer determines the color of hat that the $k$-th sage should say.

  We note that the inequality $d_1+d_2+\ldots+d_{n-1}+d_n\geq N$ holds by the definition of
  numbers $d_k$ and inequality \eqref{eq:clique-win}. This
  means that in the above strategy, the hypotheses of the sages cover all remainders modulo $N$.
  So the sages win.
\end{proof}

\begin{definition}
  The strategy of the sages is said to be \emph{precise} if for each hat
  arrangement exactly one of sage guesses is correct.
\end{definition}

\begin{corollary}
  Precise strategies exist if and only if the visibility graph is complete and
  the hat function satisfies equality
  \begin{equation}\label{eqn:sum=1}
    \frac1{a_1}+\frac1{a_2}+\ldots+\frac1{a_n}= 1.
  \end{equation}
\end{corollary}
\begin{proof}
  Let the sages act according to some strategy. If the graph contains two non-adjacent
  vertices $A$ and $B$, then we put arbitrary hats to all sages except for $A$ and $B$.
  Now the answers of $A$ and $B$ are determined by the strategy. let us give them hats
  for which their guesses are correct. With this hat arrangement, $A$, $B$
  and, possibly, someone else, guess correctly. Therefore, the strategy is not
  precise. The fact that the existing of precise strategy on complete graph is
  equivalent to equality~\eqref{eqn:sum=1} follows from the proof of
  theorem~\ref{thm:clique-win}.
\end{proof}

\subsection{Game on almost complete graphs}

\begin{definition}
  An \emph{almost complete} graph is a complete graph with one edge removed. And
  an \emph{almost clique} is an almost complete subgraph of some graph.
\end{definition}

\begin{corollary}
  Let $G$ be an almost complete graph obtained from a complete graph $K_n$ with
  vertices $A_1$, $A_2$, \dots, $A_n$ by removing the edge $A_{n-1}A_n$. Let
 the $i$-th sage get hat of one of $a_i$ colors.
  If graph $G$ is winning, then
  \begin{equation}\label{ineq:clique-minus-edge}
    \frac1{a_1}+\frac1{a_2}+\ldots+\frac1{a_n}-\frac1{a_{n-1}a_n}\geq 1.
  \end{equation}
\end{corollary}

\begin{proof}
  The fraction of total number of the arrangements for which $A_{n-1}$ or $A_n$
  guesses correctly, is equal to
  $$
  \frac1{a_{n-1}}+\frac1{a_n} -\frac1{a_{n-1}a_n}.
  $$
  Indeed, let us fix hat colors for the sages $A_1$, \dots, $A_{n-2}$. Then the answers of
  sages $A_{n-1}$ and $A_n$ are determined by the strategy. It is not difficult
  to see that there are exactly $a_{n-1}+a_n-1$ hat arrangements among
  $a_{n-1}a_n$ possible arrangements for $A_{n-1}$ and $A_n$, where
  either $A_{n-1}$ or $A_n$ (maybe both) guesses correctly. As for the other
  sages, each sage $A_k$ guesses correctly on $\frac1{a_k}$ fraction of all
  arrangements. So if inequality~\eqref{ineq:clique-minus-edge} does not hold,
  there exists a hat arrangement, where nobody guesses correctly.
\end{proof}

We call the game on almost complete graph \emph{almost precise}, if
inequality~\eqref{ineq:clique-minus-edge} turns into equality and the sages win.
In an almost precise game two sages ($A_{n-1}$ and $A_n$) guess their colors
correctly on the $\frac{1}{a_{n-1}a_n}$ fraction of all arrangements, and for all
other arrangements only one of the sages guesses.

We give a necessary condition for the game to be an almost precise game.

\begin{theorem}\label{thm:almost_precise_game}
  Let  $G$ be a complete graph with $n$ vertices $A_1$, $A_2$, \dots, $A_n$, in
  which edge~$A_{n-1}A_n$ has been removed, and $h(A_i)=a_i$, $i=1$, \dots, $n$.
  Let the game be almost precise, i.\,e.\ the following equality holds
  \begin{equation}\label{eq:clique-minus-edge}
    \frac1{a_1}+\frac1{a_2}+\ldots+\frac1{a_n} -\frac1{a_{n-1}a_n} =  1.
  \end{equation}
  Then $a_1a_2\cdots a_{n-2}$ is divisible by $a_{n-1}a_n$.
\end{theorem}

\begin{proof}
  The summands $\frac1{a_1}$, \dots, $\frac1{a_{n-2}}$ have clear probabilistic
  interpretation: $\frac1{a_i}$ is the fraction of the hat arrangements for which sage
  $A_i$ guesses correctly.

  Let $X$ be the set of hat arrangements for the first $n-2$~sages, i.\,e., in
  other words, $X$ is a collection of sets of $n-2$~colors, where the first
  color is a possible hat color of sage~$A_1$, the second color is a possible hat
  color of sage~$A_2$ and so on, the $(n-2)$-th color is a possible hat color of
  sage~$A_{n-2}$. Let $\alpha = a_1a_2\cdots a_{n-2}$, then $|X|=\alpha $.
  We split set $X$ onto subsets~$L_i$ ($i=1$, $2$, \dots, $a_{n-1}$) such that if
  sage~$A_{n-1}$ sees a set of colors from~$L_i$ on his neighbors, then he
  says color~$i$. The sets~$R_j$ ($j=1$, 2, \dots, $a_n$) for sage~$A_n$ are defined
  similarly. Let $L_k$ be the set~$L_i$ of minimum cardinality, $|L_k|=M\leq
  \frac{\alpha}{a_{n-1}}$. Now we consider the sets $R_j\setminus L_k$ ($j=1$, 2, \dots,
  $a_n$). These sets contain $\alpha-M$ elements in total and hence if $R_m\setminus
  L_k$ is the set of minimum cardinality, then $|R_m\setminus L_k|\leq
  \frac{\alpha-M}{a_n}$. Therefore,
  \begin{align}
    |L_k \cup R_m|&=|L_k|+|R_m\setminus L_k|\leq M+\frac{\alpha-M}{a_n}
    =
      \frac{\alpha}{a_n}  + M\biggl(1- \frac{1}{a_n}\biggr)
      \leq
      \frac{\alpha}{a_n}  + \frac{\alpha}{a_{n-1}}\biggl(1- \frac{1}{a_n}\biggr)
      \label{eqn:2nerav}
      =\\&=
    \alpha\biggl(\frac{1}{a_{n-1}}+ \frac{1}{a_n}-\frac{1}{a_{n-1}a_n}\biggr)
    =
    \alpha\biggl(1-\sum_{i=1}^{n-2}\frac{1}{a_{i}}\biggr)
    =
    \alpha -\frac{\alpha}{a_1}-\ldots -\frac{\alpha}{a_{n-2}}.
    \notag
  \end{align}
  Thus, if sage~$A_{n-1}$ has the hat of color~$k$, and sage $A_n$ has the hat of
  color~$m$,
  and the remaining sages have colors of hats arrangement from the set
  $X\setminus(L_k \cup R_m)$, then one of the sages $A_1$, $A_2$, \dots $A_{n-2}$
  guesses correctly. The fraction of hat arrangements, for which this 
  happens, is greater than or equal to $\rho =\frac1{a_1}+\frac1{a_2}+\ldots +\frac1{a_{n-2}}$.
  But $\rho$ bounds from above the number of
  arrangements on which the sages $A_1$, $A_2$, \dots $A_{n-2}$ win. Therefore both
  inequalities \eqref{eqn:2nerav} are equalities. Then
  $|L_k|=\frac{\alpha}{a_{n-1}}$ (moreover $|L_i|=\frac{\alpha}{a_{n-1}}$ for
  all $i$),  and $|R_ m\setminus
  L_k|=\frac{\alpha}{a_{n}}-\frac{\alpha}{a_{n-1}a_{n}}$. Analogously
  $|R_j|=\frac{\alpha}{a_{n}}$. Thus,  $|R_m \cap
  L_k|=\frac{\alpha}{a_{n-1}a_{n}}$, and $\alpha$ is divisible by
  ${a_{n-1}a_{n}}$.
\end{proof}

\begin{corollary}
  Inequality~\eqref{ineq:clique-minus-edge} is not sufficient for the sages to win
  on almost complete graphs. For $n=4$ almost complete graph with hat function
  $a_1=3$, $a_2=6$, $a_3=3$, $a_4=4$ (edge $A_3A_4$ is removed) is losing, though
  it satisfies inequality \eqref{ineq:clique-minus-edge} and
  equality~\eqref{eq:clique-minus-edge}.
\end{corollary}

It immediately follows from theorem~\ref{thm:almost_precise_game} because
$a_1a_2$ is not divisible by $a_3a_4$ here.

Now consider two cases when the conditions of theorem~\ref{thm:almost_precise_game} are not sufficient.

\begin{lemma}
  Let $G$ be an almost complete graph on vertices $A_1$, \dots, $A_n$, $n\geq 4$,
  and without the edge $A_{n-1}A_n$. Let hat function $h$ satisfy the
  equality~\eqref{eq:clique-minus-edge}. Then
  \begin{enumerate}[1), nosep]
  \item if $h(A_1)=2$, then the sages lose.
  \item if $h(A_1)=3$, $h(A_n)=2$, then the sages lose.
  \end{enumerate}
\end{lemma}

\begin{proof}
  Let the sages fix some strategy.

1) Let the adversary give to the sages $A_2$, \dots, $A_{n-2}$ an arbitrary
    collection of hats. We determine which color sage $A_{n-1}$ says in accordance with his
    strategy when the hat of $A_1$'s is of color 0, and give the hat of this
    color to $A_{n-1}$. Analogously we determine which color sage $A_{n}$ says in accordance with his
    strategy when the hat of $A_1$ is of color~1 and give the hat of this
    color to $A_{n-1}$. Thus we give the hats to all the sages except $A_1$, and hence the guess of
    the sage $A_{1}$ is now determined. We can give him a hat of the color
    he guesses correctly. Then two sages seeing each other guess
    correctly. But this is impossible in almost precise games.

2) Let the adversary give to the sages $A_2$, \dots, $A_{n-2}$ an arbitrary
    collection of hats. There are hats of three colors for $A_1$. Let us consider
    the guess of $A_n$ in accordance with the strategy for each of these three colors.
    In two of the cases, the sage $A_n$ says the same color and we give him the hat of this
    color. If in the third case $A_n$ says another color, then we give
    $A_{n-1}$ the hat of the color which  $A_{n-1}$ says in this case (so
    he will guess). Otherwise, give to $A_{n-1}$ an arbitrary hat. Thus regardless
    of the color $A_{1}$'s hat, one of the sages $A_{n-1}$, $A_{n}$ 
    guesses correctly. Now the answer of $A_1$ in accordance with the strategy is determined.
    We give $A_1$ a hat the color of which he guesses correctly. Thus
    two sages seeing each other guess correctly, which is impossible.
\end{proof}

For example, the sages lose on almost complete graph on $4$ vertices $A_1$,
$A_2$, $A_3$, $A_4$ (edge $A_3A_4$ is absent), where $h(A_1)=2$, $h(A_2)=10$,
$h(A_3)=4$, $h(A_4)=5$.

Finally, we demonstrate an example, where the
equality~\eqref{eq:clique-minus-edge} holds and almost precise game is possible.

\begin{lemma}
  Let $G$ be an almost complete graph on 4 vertices $A$, $B$, $C$, $D$, in which
  edge~$CD$ has been removed. Let $h(A)=6$, $h(B)=6$, $h(C)=2$, $h(D)=3$. Then
  the sages win.
\end{lemma}

\begin{proof}
  We interpret the hat colors of sages $A$ and $B$ as integers modulo~6,  color of~$C$
  as integer modulo~2,  color of~$D$ as integer modulo~3. Denote the hat colors of
  sages $A$ and $B$ by $a$ and~$b$. Let sage $C$ say color $c= (a+b) \bmod 2$,
  let sage $D$ say color $d= (a+b) \bmod 3$. If sages $C$ and $D$ did not guessed
  correctly,then $a+b=c+1 \bmod 2$  holds and also either $a+b=d+1 \bmod 3$ or $a+b=d+2 \bmod 3$. Then let $A$ compute
  his own color assuming that $a+b=c+1 \bmod 2$  and  $a+b=d+1 \bmod 3$, and $B$
  compute his color assuming that $a+b=c+1 \bmod 2$  and  $a+b=d+2 \bmod 3$.
\end{proof}

This result can also be obtained by using constructor of theorem~\ref{thm:addA2to_edge}.

\subsection{Maximum number of hats}

We present a funny corollary of theorem~\ref{thm:clique-win}. We ask
what is the maximum number of hats given to a sage in a winning graph
on $n$ vertices? To make the question meaningful we require that the hat function
makes the graph \emph{simple}, i.\,e.\ for each its subgraph the sages do not win
on this subgraph. Obviously, it is sufficient to find the maximum number for complete
graphs.

So the question is equivalent to the following number-theoretical combinatorial
problem: given~$n$, find  $\max(a_1, a_2, \dots, a_n)$, where the positive
integers $a_i$ satisfy relation~\eqref{eq:clique-win}. The solution of this problem is known, namely,
this maximum is determined by Sylvester's sequence $(s_n)$:
$$
s_0=2, \qquad s_n=1+\prod\limits_{i=0}^{n-1}s_i.
$$
and $\max(a_1, a_2, \dots, a_n)=s_n-1$. The proof can be found
in~\cite{Soundararajan2005}.

Sylvester's sequence grows very quickly, for example, $s_8$ is a 27-digit number. Thus
if 8 sages are going to win in the \hats game on the complete graph then 
one of them can be given 27-digit number of hats! In recreational mathematics the
phrases ``number 8'' and ``large numbers'' are associated with the story about
of the inventor of chess, who asked to be given $2^{64}-1$ wheat grains as reward. The
number $2^{64}-1$ has ``only'' 21 digits. In fairness, we note that both sequences
grow as $C^{2^n}$, where $C$ is a constant.

\section{Constructors}

In this section we describe several constructors.
Each constructor is a theorem providing a tool which allows to construct
new winning games by combining several graphs in a new graph.

\subsection{Product}

\subsubsection{Constructor ``Product''}\label{subsubsec:product}

\begin{definition}
  Let $A\in V(G)$. We say that a winning graph satisfies the \emph{maximum condition
    in vertex~$A$} if increasing the hatness of vertex $A$ by 1 makes the graph
  losing.

  Let $G_1=\langle {V_1, E_1} \rangle$, $G_2=\langle {V_2, E_2} \rangle$ be two
  graphs sharing common vertex $A$. \emph{The sum of graphs $G_1$, $G_2$ with
  respect to vertex $A$} is the graph $\langle {V_1\cup V_2, E_1\cup E_2} \rangle$. 
  The sum is denoted by  $G_1\gplus{A}G_2$.

  Let $\HG_1 =\left\langle G_1, h_1 \right\rangle$, $\HG_2 =\left\langle G_2, h_2
  \right\rangle$ be two games such that $V_1\cap V_2 = \{A\}$. The game  $\HG =
  \langle {G_1 \gplus{A} G_2, h} \rangle$, where $h(v)$ equals $h_i(v)$ for $v\in
  V(G_i)\setminus\{A\}$ and $h(A)=h_1(A) \cdot h_2(A)$
  (fig.~\ref{fig:multiplication}), is called \emph{a product} of games  $\HG_1$,
  $\HG_2$ with respect to vertex $A$. This product is denoted by $\HG_1\gtimes{A}
  \HG_2$.

  In such constructions, it is convenient to define the color of vertex $A$ as a
  pair $(c_1, c_2)$, where $0\leq c_1\leq h_1(A)-1$, $0\leq c_2\leq h_2(A)-1$.
  In this case, wWe say that $A$ has \emph{composite color}.
\end{definition}

\begin{theorem}[on the product of games]\label{thm:multiplication}
  Let $\HG_1 = \langle {G_1^A, h_1} \rangle$ and $\HG_2 = \langle {G_2^A, h_2}
  \rangle$ be two games such that $V(G_1)\cap V(G_2)=\{A\}$. If the sages win in
  games $\HG_1$ and $\HG_2$, then they also win in game  $\HG = \HG_1\gtimes{A}
  \HG_2$.
\end{theorem}

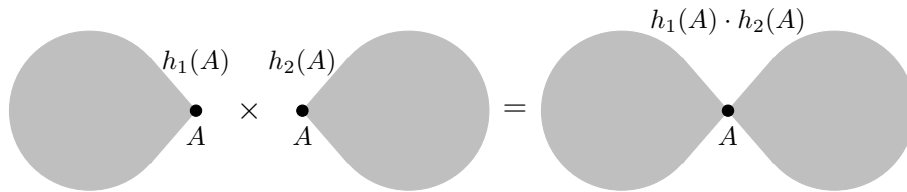
\begin{figure}[h]
\kern-1.7cm
  \centering
    \begin{tikzpicture}[
      scale=.7,
      every label/.append style = {font=\footnotesize},
      every node/.style = win_nodes,
      every edge/.append style = win_edges
      ]
      \foreach [count=\i] \x in {0, 6, 10, 14}
      \node[circle,draw,lightgray] (c\i) at (\x, 0) [minimum size=60pt] {};

      \foreach [count=\i] \x in {2, 4, 12, 12}
      {
        \coordinate (a\i) at (\x, 0);

        \draw[lightgray, fill=lightgray] (a\i) --
        (tangent cs:node=c\i,point={(a\i)},solution=1) --
        (tangent cs:node=c\i,point={(a\i)},solution=2) -- cycle;
      }

      \begin{scope}[every node/.style = {}]
        \node at (3, 0) {{$\times$}};
        \node at (8, 0) {{$=$}};
      \end{scope}

      \node[black, label=above:$h_1(A)$, label=below:$A$] (v) at (2, 0) {};
      \node[black, label=above:$h_2(A)$, label=below:$A$] (u) at (4, 0) {};
      \node[black, label=above:$h_1(A)\cdot h_2(A)$, label=below:$A$] (vu) at (12, 0) {};
    \end{tikzpicture}
  \caption{The product of games.}
  \label{fig:multiplication}
\end{figure}

\begin{proof}
  Let the hat of sage $A$ have composite color $(c_1,c_2)$, where $c_i$ is the
  hat color of~$A$ in game~$\HG_i$. We fix winning strategies for games $G_1$ and
  $G_2$ and construct strategy for the game $\HG_1\gtimes{A} \HG_2$. Let all the sages
  in the graph $G_i\setminus\{A\}$ play according to the winning strategy for the game  $\HG_i$
  (the neighbors of $A$ in $G_i$ look only at the component $c_i$ of the 
  composite color of $A$). The sage~$A$ plays in accordance with both strategies by giving composite
  answer $(c_1,c_2)$; where $c_i$ ($i=1,2$) corresponds to his winning strategy
  for the game~$\HG_i$  (for calculating the answer, sage~$A$ looks only on his
  neighbors in the graph~$G_i$).

  The presented strategy is winning because either someone from~
  $G_1\setminus\{A\}$  or from~$G_2\setminus\{A\}$  will guess correctly, or
  $A$ guesses correctly both components of his color.
\end{proof}

\begin{corollary}\label{cor:tree_with_exp_colors}
  Let graph $G$ be a tree. The sages win in the game $\langle G, h\rangle$, where $h(v)=2^{deg(v)}$.
\end{corollary}

\begin{proof}
  The sages win in the classical game $\langle {P_2, 2} \rangle$. Multiplying
  $|E(G)|$ copies of this game, we get the required result.
\end{proof}

Corollary~\ref{cor:tree_with_exp_colors} was proved in~\cite[theorem
11]{bosek19_hat_chrom_number_graph} by induction.

In the sequel, We use the following notation for a hat function taking constant value on
the whole graph except for several vertices. Let $A$, $B$, $C$ be some vertices
of the graph. The notation $h_4^{A2B2C3}$ represents a hat function, for which $h(A)=2$,
$h(B)=2$, $h(C)=3$ (superscript) and $h(V)=4$ for all other $v\in V(G)$ (subscript).

The following corollary is an important special case of the previous one.

\begin{corollary}\label{cor:path2442}
  The sages win in the game $\langle {P_n(AB), h_4^{A2B2}} \rangle$.
\end{corollary}

We note, by the way, that together with theorem~\ref{thm:multiplication}
this corollary  is strengthened analog of lemma on ``pushing a hint''~\cite[lemma
10]{Kokhas2018}. Namely, if we consider the hatness 2 of vertex $A$ as a hint,
which bounds the number of colors for sage $A$ (there should be 4 colors, but we
simplify the game for this sage), then we can ``push'' this hatness 2 along path
$AB$, where all the other sages have hatness $4$. As a result, we see that in graph
$G_1\gplus{A} P_n(AB)$ this hatness 2 ``moves'' from vertex $A$ to vertex $B$.

\subsubsection{Non-maximality of products}

Theorem~\ref{thm:multiplication} shows, that when we stick together two winning
graphs $G_1$ and $G_2$ by vertex $A$, we can greatly increase the hatness of
vertex $A$, so that the obtained game is still winning. It is natural to
assume, that the initial games are simple. In the following example we can even
more increase $h(A)=h_1(A)h_2(A)$ keeping the graph winning. This works even in the
case when both graphs $G_i$ satisfy the maximum condition in vertex $A$ (and
in all other vertices, too).

Let graphs $G_1$ and $G_2$ be complete graphs with 5 vertices and hatnesses 4,
5, 5, 5, 6. These graphs are winning by theorem~\ref{thm:clique-win}.
If we increase the hatness of any vertex, then inequality~\eqref{eq:clique-win} is violated. So
these graphs satisfy the maximum condition in all vertices. Let $A$ be the
vertex with hatness $6$ in both graphs. Let $\HG_{37}=\langle {G_1\gplus{A} G_2,
  h}\rangle$, where $h$ has the same values for all vertices except $A$ as in
the initial graphs, and $h(A)=37=6\cdot 6 + 1$ (fig.~\ref{fig:big-bow}).

\begin{theorem}\label{thm:bantik}
  The sages win in game $\HG_{37}$.
\end{theorem}

\begin{figure}[h]%
\setlength{\unitlength}{.7mm}\footnotesize
\hfill
\begin{minipage}[b]{0.4\linewidth}
\setlength{\unitlength}{.35mm}
\begin{center}
\begin{picture}(100,70)(0,-5)
\put(0,20){\circle*{4}}\put(0,40){\circle*{4}}
\put(30,0){\circle*{4}}\put(30,60){\circle*{4}}\put(50,30){\circle*{4}}
\put(70,0){\circle*{4}}\put(70,60){\circle*{4}}
\put(100,20){\circle*{4}}\put(100,40){\circle*{4}}
\put(0,20){\line(3,-2){30}}
\put(0,20){\line(0,1){20}}\put(100,20){\line(0,1){20}}
\put(0,40){\line(3,2){30}}
\put(30,60){\line(2,-3){40}}
\put(30,0){\line(2,3){40}}
\put(70,0){\line(3,2){30}}\put(70,60){\line(3,-2){30}}
\put(0,20){\line(5,1){100}}\put(0,40){\line(5,-1){100}}
\put(0,20){\line(3,4){30}} \put(0,40){\line(3,-4){30}}
\put(30,0){\line(0,1){60}} \put(70,0){\line(0,1){60}}
\put(70,0){\line(3,4){30}} \put(70,60){\line(3,-4){30}}

\put(-8,13){5}  \put(-8,37){5}
\put(103,13){5}\put(104,37){5}
\put(34,-5){5} \put(34,58){4}
\put(61,-5){5} \put(61,58){4}
\put(45,15){37}

\end{picture}
\caption{Game $\HG_{37}$ ``Big bow''.}%
\label{fig:big-bow}%
\end{center}
\end{minipage}
\hfill
\begin{minipage}[b]{0.4\linewidth}
\begin{center}
\begin{picture}(20,28)(0,-5)
\put(0,0){\circle*{2}}\put(10,0){\circle*{2}}
\put(0,10){\circle*{2}}\put(10,10){\circle*{2}}\put(20,10){\circle*{2}}
\put(10,20){\circle*{2}}\put(20,20){\circle*{2}}
\put(0,0){\line(1,0){10}}\put(0,0){\line(1,1){20}}\put(0,0){\line(0,1){10}}
\put(10,0){\line(0,1){20}}\put(10,0){\line(-1,1){10}}
\put(0,10){\line(1,0){20}}\put(10,20){\line(1,0){10}}
\put(20,10){\line(0,1){10}}\put(20,10){\line(-1,1){10}}
\put(-1,-5){5} \put(9,-5){5} \put(21,6){5}
\put(1,11){5}  \put(11,6){5}
\put(11,21){5}  \put(21,21){5}
\end{picture}
\caption{Game ``Medium bow''.}%
\label{fig:mid-bow}%
\end{center}
\end{minipage}
\hfill\quad
\end{figure}

\begin{proof}
  We deal with integers modulo $740=4\cdot5\cdot37$. If the sage hatness equals $k$
  then the possible hat colors of this sage are defined to be the integers modulo 740, divisible
  by~$\frac{740}{k}$. Let us consider a hat arrangement. Denote by $S_1$ and $S_2$ the
  sums of the colors in the left and the right \hbox{5-cliques}. Let the
  sage of hatness~4 in the left clique assumes that the color of his hat is
  such that $S_1\in \{1,2,\ldots,185\}$.
  The possible colors of his hat are $\frac{740}4=185$, $185\cdot 2$, $185\cdot 3$ and
  $185\cdot 4\equiv 0$. For exactly one of these integers the sum $S_1$ belongs to the set
  $ \{1,2,\ldots,185\}$ and the sage says the color corresponding to this integer.
  
  Similarly, three sages with hatness~5 in the left clique check the
  hypotheses $S_1\in \{186,\ldots\, 333\} $, $S_1\in\{334,\ldots\, 481\}$, and
  $S_1\in\{482,\ldots\, 629\} $, respectively (each of these sets contains
  $148=\frac{740}5$ numbers). The remaining hat arrangements (for which
  $S_1\in\{630,\ldots, 740\}$) are left to sage $A$. His hatness equals 37 and his
  colors are residues divisible by 20. Therefore, sage $A$ has a choice of
  $\lceil (740-630+1)/20\rceil=6$ consecutive colors. The right sages acts
  similarly, but replacing $S_1$ with~$S_2$. So sage $A$ has also a choice of 6 consecutive
  colors for the $S_2$-hypothesis.

  Let us describe details of sage $A$'s strategy. Let the sages of left and
  right cliques use different rules for converting the color of sage $A$ into an
  integer. The colors of sage $A$ are in fact integers modulo 37, but we
  consider them as integers modulo~740 which are divisible by~20. Let the sages
  on the left clique convert a color $20x \bmod 740 $ to an integer $x \bmod
  37$. Let in the same moment the sages on the right clique convert a color $20x
  \bmod 740 $ to an integer $6x \bmod 37$. (The map $x\mapsto 6x$ is a bijection
  on the set of integers modulo~37.)

  As is easily seen, any sets $\{x,x+1,\ldots,x+5\}$ and $\{6y,6y+6,\ldots, 6y+30\}$ of residues modulo 37 intersect in at most
  one element. Then $A$ says the color from the intersection (or says an
  arbitrary color if the intersection is empty). Then for any $S_1$ and $S_2$,
  either $A$ guesses both sums, or someone on the right or on the left clique
  guesses.
\end{proof}

It can analogously be proved that the sages win on graph ``Medium bow'',
fig.~\ref{fig:mid-bow}. This fact disproves the hypotheses 4 and 6 from
\cite{bosek19_hat_chrom_number_graph}.

\subsection{Substitution}

The following constructor removes a vertex of a graph $G_1$ and put a graph $G_2$ on
its place.

\begin{definition}
  Let $G_1$ and $G_2$ be two graphs without common vertices. A
  \emph{substitution of graph~$G_2$ to graph~$G_1$ in place of vertex~$v$}
  is defined to be the graph $(G_1\setminus\{v\})\cup G_2$  with additional edges
  connecting each vertex of $G_2$ with each neighbor of~$v$, see
  fig.~\ref{fig:subs}. We denote the substitution by $G_1[v:=G_2]$.
\end{definition}

\begin{figure}[h]
\footnotesize
\setlength{\unitlength}{300bp}%
\begin{center}
  \begin{picture}(1,0.1290902)%
    \put(0,0){\includegraphics[width=\unitlength,page=1]{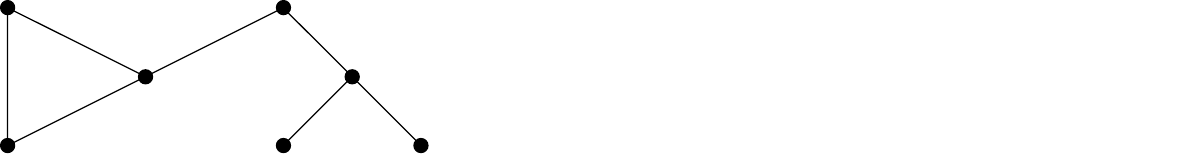}}%
    \put(0.11107517,0.03557355){\makebox(0,0)[lb]{\smash{$v$}}}%
    \put(0.48261333,0.07033853){\makebox(0,0)[lb]{\smash{$\longrightarrow$}}}%
    \put(0.45943632,0.02398464){\makebox(0,0)[lb]{\smash{$v:=$}}}%
    \put(0,0){\includegraphics[width=\unitlength,page=2]{subs.pdf}}%
  \end{picture}
\end{center}
  \caption{A substitution.}\label{fig:subs}
\end{figure}

\begin{theorem}\label{thm:substitution}
  Let the sages win in games $\HG_1=\langle {G_1, h_1} \rangle$ and $\HG_2=\langle
  {G_2, h_2} \rangle$. Let $G$ be the graph of the substitution $G_1[v:=G_2]$, where
  $v\in G_1$ is an arbitrary vertex. Then the game $\HG=\langle{G, h} \rangle$ is
  winning, where
  \[
    h(u) = \begin{cases}h_1(u)&u\in G_1\\ h_2(u)\cdot h_1(v)&u\in G_2\end{cases}
  \]
\end{theorem}

\begin{proof}
  Let $f_1$ and $f_2$ be winning strategies in the games $\HG_1$ and $\HG_2$,
  respectively.

  Let each sage $u$ of the subgraph~$G_2$ of $G$ get a hat of composite
  color~$(c_1, c_2)$, where $0\leq c_1\leq h_1(v)-1$, $0\leq c_2\leq h_2(u)-1$.
  These sages can calculate the coordinates of their composite colors independently:
  sage $v$ finds colors $c_1$ and $c_2$ by using strategies~$f_1(v)$ and~$f_2(v)$,  
  respectively. In particular, this means that all sages of~$G_2$
  say composite colors with the same first component.

  Those of the other sages of~$G$, who are not the neighbors
  of~$v$, play in accordance with the strategy~$f_1$. After the substitution the sages of~$G_1$, who are neighbors of~$v$, see that instead of one
  neighbor~$v$ they have now $|V_2|$ neighbors (and, generally speaking, with
  different hat colors). These sages do as follows. They see all the
  hats of the sages of~$G_2$ and know their strategies.
  Therefore, they understand, which of the sages of~$G_2$ 
  guesses the second coordinate of his color. Denote this player
  by~$w$ (if there are several winners, then they choose, for example,
  the first winner in the pre-compiled list).
  Then each former neighbor of~$v$ looks only at~$w$, more
  precisely, at the first component of $w$'s color, and plays
  in accordance with the strategy~$f_1$.

  As a result, either someone from subgraph $G_1 \setminus\{v\}$  guesses
  correctly, or $w$ guesses both components of his color correctly.
\end{proof}

\begin{corollary}\label{cor:triangle244-win}
  The sages win in the games shown in fig.~\ref{fig:cor:triangle244-win}.
\end{corollary}

\begin{proof}
 We apply theorem~\ref{thm:substitution} to games $\HG_1=\langle P_2, 2
  \rangle$ and $\HG_2=\langle P_n(AB), h_4^{A2B2}\rangle$.
\end{proof}

Note, that the win of the sages of the first graph
(fig.~\ref{fig:cor:triangle244-win}) also follows from
theorem~\ref{thm:clique-win}.

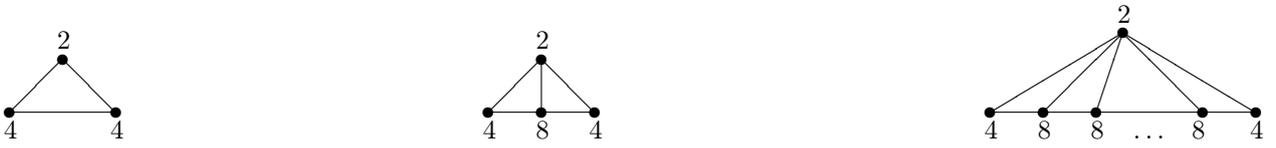
\begin{figure}[h]%
\setlength{\unitlength}{.7mm}\footnotesize

\begin{minipage}[t]{0.15\linewidth}
\begin{center}
\begin{picture}(20,20)(0,-5)
\put(0,0){\circle*{2}}\put(10,10){\circle*{2}}\put(20,0){\circle*{2}}
\put(0,0){\line(1,0){20}}\put(0,0){\line(1,1){10}}\put(10,10){\line(1,-1){10}}
\put(-1,-5){4} \put(19,-5){4} \put(9,12){2}
\end{picture}
\end{center}
\end{minipage}
\hfill
\begin{minipage}[t]{0.2\linewidth}
\begin{center}
\begin{picture}(20,20)(0,-5)
\put(0,0){\circle*{2}}\put(10,0){\circle*{2}}\put(20,0){\circle*{2}}
\put(10,10){\circle*{2}}
\put(0,0){\line(1,1){10}}\put(10,0){\line(0,1){10}}\put(20,0){\line(-1,1){10}}
\put(0,0){\line(1,0){20}}
\put(-1,-5){4} \put(19,-5){4} \put(9,-5){8}
\put(9,12){2}
\end{picture}
\end{center}
\end{minipage}
\hfill
\begin{minipage}[t]{0.3\linewidth}
\begin{center}
\begin{picture}(50,22)(0,-5)
\put(0,0){\circle*{2}}\put(10,0){\circle*{2}}\put(20,0){\circle*{2}}
\put(40,0){\circle*{2}}\put(50,0){\circle*{2}}
\put(25,15){\circle*{2}}
\put(0,0){\line(1,0){50}}
\put(0,0){\line(5,3){25}}\put(10,0){\line(1,1){15}}
\put(20,0){\line(1,3){5}}\put(40,0){\line(-1,1){15}}
\put(50,0){\line(-5,3){25}}
\put(-1,-5){4} \put(9,-5){8} \put(19,-5){8} \put(27,-5){\dots}
\put(38,-5){8} \put(49,-5){4} \put(24,17){2}
\end{picture}
\end{center}
\end{minipage}
\caption{Substitution of game $\langle P_n(AB), h_4^{A2B2}\rangle$ into game $\langle P_2, 2\rangle$.}
\label{fig:cor:triangle244-win}
\end{figure}

\subsection{Vertices attaching}

The following theorems-constructors allow to obtain new winning or losing
graphs by attaching one or two new vertices to the existing graph.

\subsubsection{Attaching a vertex of hatness 2}

\begin{theorem}\label{thm:addA2}
  Let $\langle {G, h} \rangle$ be a winning game and $B, C\in V(G)$. Then the
  sages win in the game $\langle {G', h'} \rangle$, where $G'$ is the graph obtained
  from $G$ by adding a vertex $A$ and edges $AB$, $AC$
  (fig.~\ref{fig:addA2}), and hat function is given by formula
  \[
    h'(v)=
    \begin{cases}
      2,       &\text{if}\ v = A,\\
      h(v) + 1,&\text{if}\ v = B\ \text{or}\ C,\\
      h(v),    &\text{else}.
    \end{cases}
  \]
\end{theorem}
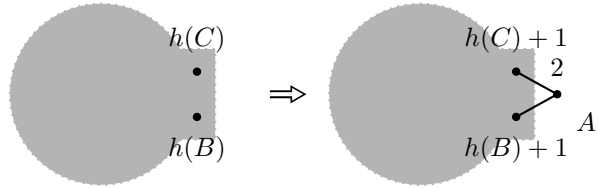
\begin{figure}[h]
  \centering\footnotesize
  \begin{tikzpicture}[scale=.6]
    \draw [fill=green!10!,lgreen,dotted,thick] (0,0) circle (2);
    \draw [fill=green!10!,lgreen,dotted,thick] (1,-1) rectangle (2.5, 1);
    \draw [fill=black] (2.1, -0.5) circle (0.08) node [label={below:{$h(B)$}}] {};
    \draw [fill=black] (2.1, 0.5) circle (0.08) node [label={above:{$h(C)$}}] {};
    \draw [vecArrow] (3.7, 0) to (4.5, 0);

    \draw [fill=green!10!,lgreen,dotted,thick] (7,0) circle (2);
    \draw [fill=green!10!,lgreen,dotted,thick] (7,-1) rectangle (9.5, 1);
    \draw [fill=black] (9.1, -0.5) circle (0.08) node [label={below:{$h(B)+1$}}] {};
    \draw [fill=black] (9.1, 0.5) circle (0.08) node [label={above:{$h(C)+1$}}] {};
    \draw [fill=black] (10, 0) circle (0.08) node [label=above:2, label=-45:$A$] {} ;
    \draw [thick] (9.1, -0.5) -- (10, 0);
    \draw [thick] (9.1, 0.5) -- (10, 0);
  \end{tikzpicture}
  \caption{Attaching a vertex of hatness $2$.}
  \label{fig:addA2}
\end{figure}

\begin{proof}
  Let us describe a winning strategy. After attaching of new vertex sages $B$ and $C$
  have one new possible color. Let sage~$A$ say ``1'' if he sees at least one
  hat of the new color on sages~$B$ and $C$, otherwise $A$
  says ``0''. If sages~$B$ and $C$ see a hat of color 0 on~$A$, then let they both
  say the new color. Thus, if $A$'s color is 0, then one of the sages $A$, $B$
  and $C$ win. If $A$'s color is 1, then $B$ and $C$ may think that
  have not a new color, and therefore can
  play by their strategies of the game~$\langle {G, h} \rangle$.
\end{proof}

\begin{corollary}\label{cor:cycle323}
  Let $G$ be cycle $C_n$ \ $(n \geq 4)$, and let $B$, $A$ and $C$ be
  three consequent vertices of the cycle. Then the sages win in the game
  $\langle {G, h_4^{B3A2C3}} \rangle$.
\end{corollary}
\begin{proof}
  By corollary~\ref{cor:path2442} the sages win on graph $P_{n-1}(CB)$ with
  hatnesses $2,4,\dots, 4,2$. Attaching vertex $A$ to this graph  gives a winning graph by theorem~\ref{thm:addA2}.
\end{proof}

This corollary strengthens lemma ``on the hint $A-1$ for cycle'' \cite[lemma
9]{Kokhas2018} without any technical calculations.

The following constructor shows that if the vertices
$B$ and $C$ in theorem~\ref{thm:addA2} are adjacent, then the numbers of colors for these vertices can greatly be increased.

\begin{theorem}\label{thm:addA2to_edge}
  Let $\langle G,h\rangle$ be a winning game, and let $BC$ be an edge of the
  graph~$G$. Consider a graph $G'=\left\langle V', E' \right\rangle$ obtained by
  adding a new vertex~$A$ and two new edges to graph~$G$: $V'=V\cup \{A\}$, $E'=E\cup
  \{AB, AC\}$. Then the sages win in the game $\langle G',h'\rangle$ (see
  fig.~\ref{fig:addA2_with_edge}), where
  \[
    h'(v) = \begin{cases}
      2,&\text{if}\ v = A\\
      2h(v),&\text{if}\ v = B \text{ or } v = C\\
      h(v),&\text{otherwise}
    \end{cases}
  \]
\end{theorem}
\begin{figure}[h]
  \centering\footnotesize
  \begin{tikzpicture}[scale=.65]
    \draw [fill=green!10!,lgreen,dotted,thick] (0,0) circle (2);
    \draw [fill=green!10!,lgreen,dotted,thick] (1,-1) rectangle (2.5, 1);
    \draw [fill=black] (2.1, -0.5) circle (0.08) node [label={below:{$h(B)$}}] {};
    \draw [fill=black] (2.1, 0.5) circle (0.08) node [label={above:{$h(C)$}}] {};

    \draw [thick] (2.1, 0.5) -- (2.1, -0.5);

    \draw [vecArrow] (3.7, 0) to (4.5, 0);

    \draw [fill=green!10!,lgreen,dotted,thick] (7,0) circle (2);
    \draw [fill=green!10!,lgreen,dotted,thick] (7,-1) rectangle (9.5, 1);
    \draw [fill=black] (9.1, -0.5) circle (0.08) node [label={below:{$h(B)\times 2$}}] {};
    \draw [fill=black] (9.1, 0.5) circle (0.08) node [label={above:{$h(C)\times 2$}}] {};
    \draw [fill=black] (10, 0) circle (0.08) node [label=above:2, label=-45:$A$] {} ;
    \draw [thick] (9.1, -0.5) -- (10, 0);
    \draw [thick] (9.1, 0.5) -- (10, 0);
    \draw [thick] (9.1, 0.5) -- (9.1, -0.5);

  \end{tikzpicture}
  \caption{Attaching a vertex of hatness $2$ to edge $BC$.}
  \label{fig:addA2_with_edge}
\end{figure}
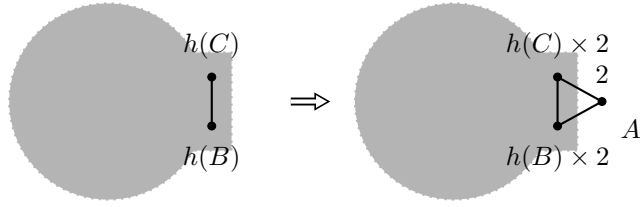
\begin{proof}
  Let the sages $B$ and $C$ have composite colors $(c,\epsilon)$, where $c$ is a
  possible hat color in the game $\langle G,h\rangle$, $\epsilon\in\{0,1\}$. Let
  sage~$A$ say color $c(A)=\epsilon_B +\epsilon_C \pmod 2$. Let sages $B$ and
  $C$ look at the colors of their neighbors in graph~$G$, calculate the  colors
  $c(B)$, $c(C)$ in accordance with their winning strategy in game $\langle G,h\rangle$ and take
  these colors as the first coordinates of their composite colors. By seeing 
  sage~$A$'s hat as well as each other  hat, the sages $B$ and $C$ can calculate
  the values $\epsilon_B$ and~$\epsilon_C$ for which $A$ does not guess correctly, they take
  these bits as second components.
\end{proof}

Attaching edge $BA$ with leaf $A$ of hatness 2 to vertex $B$ of graph $G$ can be
interpreted as the product of game on graph $G$ and classical game $\langle
{P_2(AB), 2} \rangle$. If the game on graph $G$ was winning, then by
the theorem on product, we can double hatness of vertex $B$ in construction of winning game $G\gplus{B} BA$. In the following constructor we attach the edge to
a losing game, change the hatness of vertex $B$ to $2h(B)-1$, and as a result we
get a losing game.

\begin{theorem}
  Let $\HG = \langle {G, h} \rangle$ be a loosing game and $B$ be an arbitrary
  vertex of graph~$G$. Let $G'=\left\langle V', E' \right\rangle$ be a graph 
  obtained by attaching a new pendant vertex~$A$ to graph~$G$: $V'=V\cup \{A\}$,
  $E'=E\cup \{AB\}$. Then the sages loose in the game $\langle G',h'\rangle$, where
  $h(A)=2$,  $h'(B)=2h(B)-1$ and $h'(u)=h(u)$ for other vertices $u\in V$.
\end{theorem}

\begin{proof}
  Let the sages fix a strategy $f$ on graph~$G'$. We construct a losing hat
  arrangement for this strategy. For each of $2h(B)-1$ possible colors of $B$'s,
  hat sage~$A$ has one (of 2 possible) answers in accordance with his strategy. He
  says one of these two colors less often, namely, at most $h(B) - 1$ times. Let
 the adversary give to sage $A$ the hat of this ``rare'' color. Then the strategy
  of sage $B$ in game $\HG$ is now completely determined. Let the adversary use
  only the hats of those $h(B)$ colors for sage $B$, for which $A$ does not
  guess correctly. Under this restriction, the adversary can nevertheless construct a
  losing hat arrangement on $G$, because the game $\HG$ is losing. So the
  adversary can construct the loosing hat arrangement on graph $G'$.
\end{proof}

It is possible to attach a new vertex $A$ of hatness $2$ simultaneously to several
vertices of a losing graph. If we increase hatnesses of these vertices greatly, this
cancels out the possible advantage from the appearance of a new vertex and the graph
remains losing. In the following theorem, we consider the case of two vertices.

\begin{theorem}
  Let $G$ be a loosing graph with vertices $B$ and $C$, and
  $h(B)=h(C)=2$. Attach to the graph a new vertex~$A$
  connected with~$B$ and~$C$ only. Then the sages lose on the obtained
  graph, if we define new hat function as $h(A)=2$, $h(B)=3$, $h(C)=7$, and for
  other vertices the hat function is the same as in $G$.
\end{theorem}

\begin{proof}
  Let the sages fix some strategy on the new graph. The strategy of sage~$A$ can
  be given as $3\times 7$ table: the rows correspond to the hat colors of
  sage~$B$, the columns correspond to the hat colors of sage~$A$, and the table
  entry (0 or 1) is the number that sage~$A$ says, when he sees the
  corresponding $B$'s and $C$'s hat colors.

  Each column of the table contains one of the symbols, 0 or 1, two times. Mark
  in each column two cells that contain the symbol repeated at least twice in
  this column. (If the symbol is repeated in all three cells of column, we mark any two
  of them.) The marked cells can be located either in the first and the second
  rows, or in the first and the third, or in the second and in the third. Since
  there are 7~columns, the pigeonhole principle implies that there exist two rows, in
  which the marked cells occupy three columns. The marked cells of one column
  contain either two zeroes or two ones. Therefore, one can choose two
  columns containing the same numbers in the marked cells.

  Thus we have chosen two rows (for definiteness, the $i$-th and the $j$-th) and two
  columns (for definiteness, the $k$-th and the $\ell$-th), which intersect in 4 cells
  containing the same number, say~0. Now we construct a
  disproving hat arrangement on  the new graph. First, give hthe at of color~1 to
  sage~$A$.

  Then we choose a hat of the $i$-th or the $j$-th color for sage $B$, and
  a hat of the $k$-th or the \hbox{$\ell$-th} color for sage~$C$. In this
  case, sage~$A$ says ``0'' and guesses his color incorrectly. To
  assign colors to sages $B$ and $C$, and the others, we consider a
  game on graph~$G$: since  the color of sage~$A$ is fixed, the strategies
  of the other sages on graph~$G$ is uniquely determined. The restrictions to the
  hat colors of $B$ and $C$ allow us to think that $h(B)=h(C)=2$. Since
  graph~$G$ is losing for this hat function, we will successfully find a
  disproving hat arrangement on it.
\end{proof}

The proof is based on ideas from Ramsey theory. Thus the statement can be generalized
to the case of large number of new vertices and those vertices to which they are
attached. However, apparently, such constructions give too overestimated values of the
hatnesses in losing graphs.

\subsubsection{Attaching vertices of hatnesses 2 and 3, connected by an edge}

Apparently it is hard to determine whether the graph obtained by attaching a new fragment via two independent ``jumpers'', is winning. We are able to do this for very small fragment only.

\begin{theorem}\label{thm:addA2B3}
  Let $\HG=\langle G,h\rangle$ be a winning game, and let $Z$, $C \in V$ be two
  vertices of graph~$G$. Consider a graph $G'=\left\langle V', E' \right\rangle$
  obtained by adding a new path~$ZABC$ to graph~$G$: $V'=V\cup\{A, B\}$, $E'=E\cup
  \{ZA, AB, BC\}$ (fig.~\ref{fig:addA2B3}). Then the sages win in the game
  $\HG'=\langle G',h'\rangle$, where
  \[
    h'(v)=
    \begin{cases}
      2,       &\text{if}\ v = A,\\
      3,       &\text{if}\ v = B,\\
      2h(v),   &\text{if}\ v = Z,\\
      h(v) + 1,&\text{if}\ v = C,\\
      h(v),    &\text{otherwise.}
    \end{cases}
  \]
\end{theorem}
\begin{figure}[h]\centering\footnotesize
  \begin{tikzpicture}[scale=.65
    ]
    \draw [fill=green!10!,lgreen,dotted,thick] (0,0) circle (2);
    \draw [fill=green!10!,lgreen,dotted,thick] (1,-1) rectangle (2.5, 1);
    \draw [fill=black] (2.1, -0.5) circle (0.08) node [label={below:{$h(C)$}}] {};
    \draw [fill=black] (2.1, 0.5) circle (0.08) node [label={above:{$h(Z)$}}] {};
    \draw [vecArrow] (3.7, 0) to (4.5, 0);

    \draw [fill=green!10!,lgreen,dotted,thick] (7,0) circle (2);
    \draw [fill=green!10!,lgreen,dotted,thick] (7,-1) rectangle (9.5, 1);
    \draw [fill=black] (9.1, -0.5) circle (0.08) node [label={below:{$h(C) + 1$}}] {};
    \draw [fill=black] (9.1, 0.5) circle (0.08) node [label={above:{$h(Z) \times 2$}}] {};
    \draw [fill=black] (10, 0.5) circle (0.08) node [label=right:2, label=60:$A$] {};
    \draw [fill=black] (10, -0.5) circle (0.08) node [label=right:3, label=-60:$B$] {};
    \draw[thick] (9.1, -0.5) -- (10, -0.5);
    \draw[thick] (9.1, 0.5) -- (10, 0.5);
    \draw[thick] (10, -0.5) -- (10, 0.5);
  \end{tikzpicture}

  \caption{Adding a new path~$ZABC$.}
  \label{fig:addA2B3}
\end{figure}
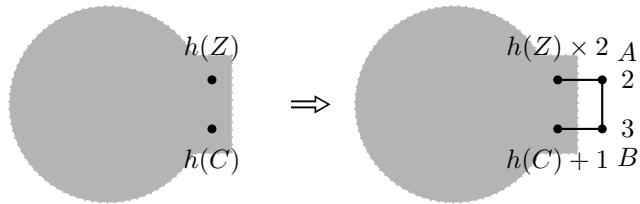

\begin{proof}
  Let sage $X$ get a hat of color~$c_X$. Consider the color $c_Z$ as
  composite:  $c_Z=(\epsilon, c)$, where $\epsilon_Z\in \{0, 1\}$, and 
  $c$ is one of $h(Z)$ colors in the game $\HG$. Let us describe a winning strategy.

  \begin{itemize}[nosep]
  \item If $c_B\ne 2$, let sage $A$ say  $B$'s hat color, otherwise he says
    the color~$\epsilon_Z$.
  \item Sage $B$ says ``2'' if he sees the hat of new color on sage~$C$,
    otherwise he says $1 - c_A$.
  \item  If $c_B\ne 2$, then sage $C$ says new color, otherwise $C$ uses strategy of the game $\HG$.
  \item Let $Z$ take the first bit of his color $\epsilon_Z\ne c_A$, and find
    the second component of his color in accordance with his strategy in $\HG$.
  \item The sages in $V(G)\setminus\{C,Z\}$ use strategy of $\HG$. It is assumed here that 
    the neighbors of $Z$ see at the second component of $Z$'s color only and
    the neighbors of $C$ do not distinguish the new color and 0-th color.
  \end{itemize}

  Now we consider all variants of pairs $(c_A, c_B)$ and check that the strategy is
  winning.

  In cases $(0, 0)$ and $(1, 1)$ sage $A$ guesses correctly.

  In cases $(0, 1)$ and $(1, 0)$ $B$ or $C$ guesses correctly.

  In cases  $(0, 2)$ and $(1, 2)$ sage $A$ guesses correctly if $c_A= \epsilon_Z$, and
  sage $B$ guesses correctly if $C$ has hat of new color. In the remaining cases, the
  sages on graph $G$ use the strategy of game $\HG$, and one of them guesses correctly.
\end{proof}

\begin{corollary}\label{cor:cycle_233}
  Let graph $G$ be cycle $C_n$ $(n \geq 4)$ and $A$, $B$, $C$ be three
  consequent vertices of the cycle. Then the sages win in game $\langle{G,
    h_4^{A2B3C3}}\rangle$.
\end{corollary}

\begin{proof}
  Follows from theorem~\ref{thm:addA2B3} applied to path $P_{n-2}(ZC)$ with hatnesses
  $2,4,\dots, 4,2$ (this game is winning by corollary~\ref{cor:path2442}).
\end{proof}

\subsubsection{Attaching a leaf of large hatness}

Our latest constructor is great in that it works both for winning
and for losing games. It claims that attaching to a graph a leaf with hatness
more than $2$ does not affect the result of the game.

\begin{definition}
  Let $\HG=\langle {G, h} \rangle$ be a game and $A\in V(G)$. 
  The \emph{game $\HG$ with the hint $A^*$} is defined as follows. The sages play on graph $G$ with
  hat function $h$, but \textbf{during the test} the adversary comes up to
  sage $A$ and says the true statement ``I just put on your head a hat of color
  $c_1$ or $c_2$''. During the conversation the sages know that the adversary
  is going to give a hint, but do not know what colors he will say. So the
  sages determine usual strategies for everyone, except for the sage $A$, and sage $A$
  gets the set of $\binom{h(A)}{2}$ strategies, one for each possible hint.
\end{definition}

``A theory of \hats game with hints'' (where hatness of each sages is equal to
3) is developed by Kokhas and Latyshev's in~\cite{Kokhas2018}. The
following lemma from~\cite{Kokhas2018} remains almost unchanged in the case
of arbitrary hatnesses. We present here its the proof to make the paper self-contained.

\begin{lemma}\label{lem:hint_Astar}
 The hint $A^*$ does not affect the result of the \hats game.
\end{lemma}
\begin{proof}
  Assume that the sages win with hint $A^*$. Let us fix strategies of all sages except for $A$
  in the game with the hint $A^*$; we construct a
  strategy of $A$ so that the sages win without hints.

  Assume that the adversary gives a hat of color $x$ to $A$, so that
  there exists a~hat arrangement in which $A$ gets a hat of color $x$, his
  neighbors get hats of colors $u$, $v$, $w$ \dots, the other sages also get
  hats of some colors, and no one (except for $A$) guesses correctly. we
  want that in this case the sage $A$ to guesses the color of his hat correctly, i.\,e., his
  strategy satisfies the requirement $f_A(u, v, w, \dots) = x$.

  These requirements for different hat arrangements do not contradict each other.
  Indeed, if there exists another hat arrangement where the neighbors still have
  colors $u$, $v$, $w$, \dots and  the sage $A$ gets another color $y$, then the
  sages cannot win with the hint $A^*$, because the adversary can inform $A$
  that  he has a hat of color $x$ or $y$ and then choose one of the two hat
  arrangements for which sage~$A$ does not guess his color correctly.
\end{proof}

\begin{theorem}
  Let $\HG_1 = \langle {G_1, h_1} \rangle$, $B\in V(G_1)$, $G_2=G_1\gplus{B} P_2(AB)$ and $\HG_2 = \langle G_2, h_2 \rangle$,
  where $h_2\Big\vert_{V(G_1)}=h_1$, \ $h_2(A)\geq 3$. The ame  $\HG_1$ is winning
  if and only if $\HG_2$ is winning.
\end{theorem}

\begin{proof}
  In one direction the statement is obvious: if game $\HG_1$ is winning, then
  game~$\HG_2$ is also winning. (The sages on the subgraph $G_1$ win.)

  Now prove that if game $\HG_2$ is winning, then $\HG_2$ is also winning. We
  demonstrate that if the sages win in game~$\HG_2$, then they can win in
  game~$\HG_1$ with hint~$B^*$.

  Let $f_2$ be a winning strategy for game $\HG_2$. In order to construct a
  winning strategy for game $\HG_1$ with the hint~$B^*$, we first 
  define a strategy for sages on $V(G_1)\setminus\{B\}$ --- let they use $f_2$.
  Second, for any two different colors $(b_1, b_2)$ that can occur in the
  hint we define a strategy of~$B$. Since $h_2(A)\geq 3$, for each pair of
  colors $(b_1, b_2)$, $b_1 \neq b_2$, we can find a color $a$ such that $A$ can
  not say ``$a$'' if he sees that $B$'s hat is of color $b_1$ or~$b_2$. Let
  sage $B$, having received the hint $(b_1, b_2)$, say the color defined by the strategy
  $f_2$ when he sees the hat of color~$a$ on sage~$A$ and the
  colors of the other neighbors in $G_1$ are given by the current arrangement.

  This strategy is winning in the game~$\HG_1$ with the hint~$B^*$. Indeed, let
  the hat arrangement on $G_1$ be fixed and sage $B$ get a hint $(b_1, b_2)$.
  Consider the corresponding hat arrangement on $G_2$ (we give the hat of color $a$ to
  sage $A$). Then all the sages on $G_2$ use strategy $f_2$ (and sage $A$ does
  not guess). Therefore, someone on $G_1$ guesses correctly. Thus, for the hat
  arrangement and the hint under consideration, the sages on $G_1$ win, and hence the sages
  win with the hint~$B^*$. Then they win in game
  $\HG_1$ by lemma~\ref{lem:hint_Astar}.
\end{proof}

The proven theorem has an interesting generalization for losing graphs. If we
glue two losing graphs $G_1$ and $G_2$ by identifying a
vertex of hatness 2 of graph $G_2$ and any vertex of graph $G_1$, then the obtained
graph is losing.

\begin{theorem}
  Let $G_1$ and $G_2$ be graphs such that $V(G_1)\cap  V(G_2)=\{A\}$,  $G=  G_1+_AG_2$.
  Assume that the games $\HG_1=\langle G_1, h_1\rangle$ and $\HG_2=\langle
  G_2, h_2\rangle $ are losing, and also $h_1(A)\geq h_2(A)=2$. Then the
  game $\HG=\langle G_1+_AG_2, h\rangle$ is losing, where
  $$
  h(x) =
  \begin{cases}
    h_1(x),&x\in V(G_1)\\
    h_2(x),&x\in V(G_2)\setminus A.
  \end{cases}
  $$
\end{theorem}

\begin{proof}
  Assume that game $\HG$ is winning, and let $f$ be a winning strategy. Let
  $N(A)$ be the set of neighbors of vertex $A$ in graph $G_1$. Every hat
  arrangement $\varphi$ on graph $G_1$ determines answer in accordance with strategy $f$ of 
  each sage in $V(G_1)\setminus A$. Let us prove that there exist two different hat
  arrangements $\varphi_1$ and $\varphi_2$ on vertices of graph $G_1$, such that
  $\varphi_1\big|_{N(A)}=\varphi_2\big|_{N(A)}$, $\varphi_1(A)\neq\varphi_2(A)$
  and if the sages of $G_1$ play in accordance with strategy $f$, then in both
  arrangements none of $V(G_1)\setminus A$ guesses correctly.

  Assume that there are no two such arrangements. This means that for every
  hat arrangement $c$ on $N(A)$, there is at most one color $\alpha(c)$ of $A$'s
  hat, for which the hat arrangement $c\cup\alpha(c)$ on $N(A)\cup A$ can be
  extended to a hat arrangements on $V(G_1)$ such that none in
  $V(G_1)\setminus A$ guesses correctly in accordance with strategy $f$. We consider the following
  strategy for game $\HG_1$. Let everyone in $V(G_1)\setminus A$ play in accordance with
  strategy $f$, and sage $A$ say $\alpha(c)$ (or 0, if $\alpha(c)$ is
  undefined). This strategy is winning, because if none in $V(G_1)\setminus
  A$ guess correctly, then $A$ has a hat of color $\alpha(c)$ and guesses correctly, a
  contradiction.

  Let us consider these two arrangements $\varphi_1$ and~$\varphi_2$. We fix the hat
  arrangement $c=\varphi_1\big|_{N(A)}=\varphi_2\big|_{N(A)}$ on $N(A)$ and
  restrict ourselves only to those arrangements, where sage $A$ gets a hat of
  one of two colors $\varphi_1(A)$ or $\varphi_2(A)$. Then strategy $f$
  determines actions of the sages on graph $G_2$, i.\,e.\, in the losing game
  $\HG_2$ (sage $A$ can say more than two colors, but it does not help to win). ``Losing''
  means that there exists a disproving hat arrangement~$\psi$. If
  $\psi(A)=\varphi_1(A)$, then $\psi\cup \varphi_1\big|_{V(G_1)\setminus A}$ is
  a disproving arrangement on~$G$, and if $\psi(A)=\varphi_2(A)$, then $\psi\cup
  \varphi_2\big|_{V(G_1)\setminus A}$ is a disproving arrangement on~$G$.
\end{proof}

\subsection{More complicated constructors}

The next theorem generalizes of theorem~\ref{thm:addA2}.

\begin{theorem}\label{thm:addA2gen}
  Let $\langle G_1 ,h_1\rangle$, $\langle G_2 ,h_2\rangle$ be two games, in
  which the sages win. Let $A_1$, $A_2$, \dots, $A_k\in V_1$ and $B_1$, $B_2$, \dots,
  $B_m\in V_2$. Let $G'=\left\langle V', E' \right\rangle$ be a graph obtained
  by adding all the edges $A_iB_j$ to graph $G_1\cup G_2$: $V'=V_1\cup V_2$,
  $E'=E_1\cup E_2\cup \{A_iB_j, i=1,\dots, k; j=1,\dots, m\}$
  (fig.~\ref{fig:addA2gen}). Then the sages win in the game $\langle G',h'\rangle$,
  where
  $$
  h'(u)=\begin{cases}
    h_1(u), & u\in G_1\setminus \{ A_1, A_2, \dots, A_k\}, \\
    h_2(u), & u\in G_2\setminus \{ B_1, B_2, \dots, B_m\}, \\
    h_1(u)+1,& u\in \{ A_1, A_2, \dots, A_k\}, \\
    h_2(u)+1,& u\in \{ B_1, B_2, \dots, B_m\}.
  \end{cases}
  $$
\end{theorem}

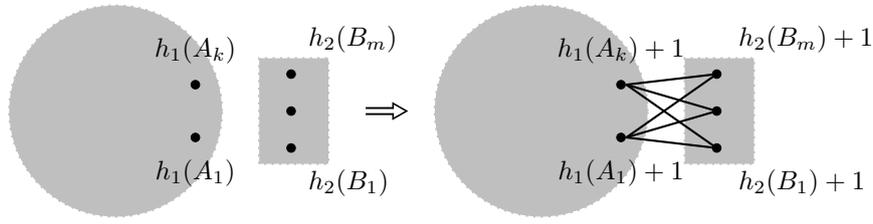
\begin{figure}[h]
 \centering
  \begin{tikzpicture}[scale=.7, every label/.append style={font=\footnotesize}]
    \draw [fill=lightgray!10!,lightgray,dotted,thick] (-1,0) circle (2);
    \draw [fill=lightgray!10!,lightgray,dotted,thick] (1.7,-1) rectangle (3, 1);
		\draw [fill=black] (.5, -0.5) circle (0.08) node [label={below:{$h_1(A_1)$}}] {};
		\draw [fill=black] (.5, 0.5) circle (0.08) node [label={above:{$h_1(A_k)$}}] {};
    \draw [fill=black] (2.3, -0.7) circle (0.08) node [label={-60:{$h_2(B_1)$}}] {};
		\draw [fill=black] (2.3, 0.) circle (0.08) node [] {};
    \draw [fill=black] (2.3, 0.7) circle (0.08) node [label={60:{$h_2(B_m)$}}] {};
    \draw [vecArrow] (3.7, 0) to (4.5, 0);

    \draw [fill=lightgray!10!,lightgray,dotted,thick] (7,0) circle (2);
    \draw [fill=lightgray!10!,lightgray,dotted,thick] (9.7,-1) rectangle (11, 1);
		\draw [fill=black] (8.5, -0.5) circle (0.08) node [label=below:{$h_1(A_1)+1$}] {};
		\draw [fill=black] (8.5, 0.5) circle (0.08) node [label=above:{$h_1(A_k)+1$}] {};
    \draw [fill=black] (10.3, -0.7) circle (0.08) node [label=-45:{$h_2(B_1)+1$}] {};
		\draw [fill=black] (10.3, 0.) circle (0.08) node [] {};
    \draw [fill=black] (10.3, 0.7) circle (0.08) node [label=45:$h_2(B_m)+1$] {};
    \draw [thick] (8.6, -0.5) -- (10.3, 0);
    \draw [thick] (8.6, 0.5) -- (10.3, 0);
		\draw [thick] (8.6, -0.5) -- (10.3, 0.7);
		\draw [thick] (8.6, 0.5) -- (10.3, 0.7);
		\draw [thick] (8.6, -0.5) -- (10.3, -0.7);
		\draw [thick] (8.6, 0.5) -- (10.3, -0.7);
  \end{tikzpicture}
  \caption{Stitching of two graphs, $k=2$, $m=3$.}
  \label{fig:addA2gen}
\end{figure}
\begin{proof}
  One new color has been added for sages $A_i$ and one for sages $B_j$ with
  respect to initial games. Let this color be red. For each $i$ let sage~$A_i$
  say that he has red hat, if he sees at least one red hat on sages~$B_j$, in
  the opposite case let $A_i$ see at his neighbors in $G_1$ only and play in accordance with
  winning strategy on~$G_1$. For each $j$, if sage~$B_j$ sees at least one
  red hat on~$A_i$, then he sees at his neighbors in $G_2$ only and plays in accordance with
  winning strategy on~$G_2$. In the opposite case $B_j$ says that he has
  red hat. It is easy to check that this strategy is winning.
\end{proof}

The next theorem add ``surgical intervention'' to the previous construction: we
will sew together graphs, by joining the neighbors of the two chosen vertices
with hatness 2, and delete both vertices.

\begin{theorem}
  Let $G_1$, $G_2$ be graphs containing vertices $A$ and $B$ respectively,
  the games $\HG_1=\langle G_1, h_1 \rangle$ and $\HG_2=\langle G_2, h_2 \rangle$ be
  winning and $h_1(A)=h_2(B)=2$. Let $N_A$ and~$N_B$ be the sets of neighbors of $A$
  and $B$ in graphs $G_1$ and $G_2$. Consider a new graph $G$
  (fig.~\ref{fig:AB-sewing}):
  \begin{gather*}
    V(G)=(V(G_1)\setminus A) \cup (V(G_2)\setminus B),
    \\
    E(G)=E(G_1)\big|_{V(G_1)\setminus A} \cup \,E(G_2)\big|_{V(G_2)\setminus B}
    \cup \,\{XY \mid X \in N_A, Y \in N_B\}.
  \end{gather*}
  Then the game $\HG=\langle G, h \rangle$ is winning, where
  \[
    h(x) =\begin{cases}
      h_1(x),&x\in V(G_1)\setminus A,\\
      h_2(x),&x\in V(G_2)\setminus B.
    \end{cases}
  \]
\end{theorem}

\begin{figure}[h]%
\begin{center}
\setlength{\unitlength}{400bp}\footnotesize
 \begin{picture}(1,0.2660933)%
    \put(0,0){\includegraphics[width=\unitlength,page=1]{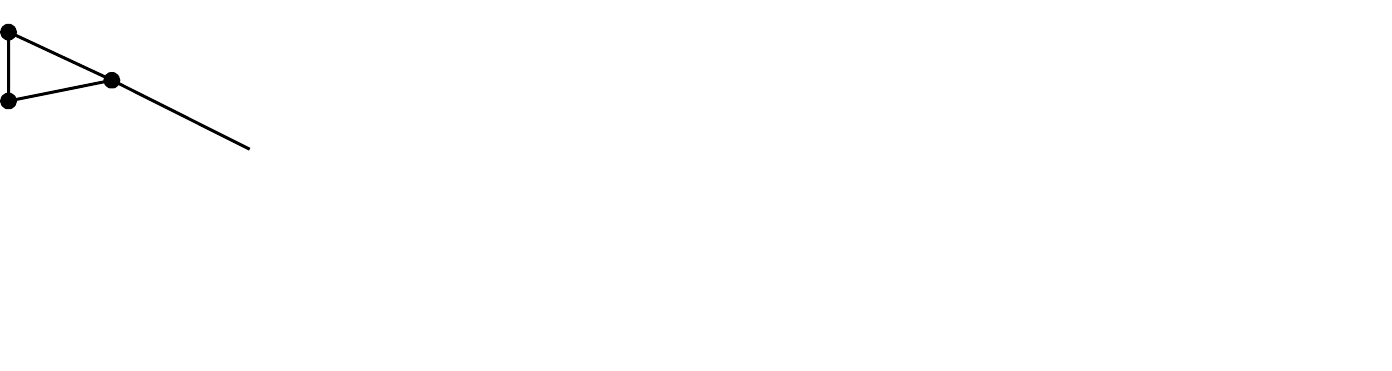}}%
    \put(0.18933233,0.13902369){$B$}%
    \put(0.17447552,0.17368627){2}%
    \put(0.0853397,0.22320809){$Y_1$}%
    \put(0,0){\includegraphics[width=\unitlength,page=2]{shivaem.pdf}}%
    \put(0.08038578,0.08){$Y_2$}%
    \put(0.05067469,0.00531748){Graph $G_2$}%
    \put(0,0){\includegraphics[width=\unitlength,page=3]{shivaem.pdf}}%
    \put(0.30191757,0.13902369){{\smash{\textit{A}}}}%
    \put(0.31677188,0.17368627){{\smash{2}}}%
    \put(0.36134224,0.24796772){$X_3$}%
    \put(0,0){\includegraphics[width=\unitlength,page=4]{shivaem.pdf}}%
    \put(0.3662937,0.17368627){$X_2$}%
    \put(0,0){\includegraphics[width=\unitlength,page=5]{shivaem.pdf}}%
    \put(0.38115051,0.05979076){$X_1$}%
    \put(0.38115051,0.00531748){{\smash{Graph $G_1$}}}%
    \put(0,0){\includegraphics[width=\unitlength,page=6]{shivaem.pdf}}%
    \put(0.85028149,0.24796772){$X_3$}%
    \put(0,0){\includegraphics[width=\unitlength,page=7]{shivaem.pdf}}%
    \put(0.85523288,0.17368627){$X_2$}%
    \put(0,0){\includegraphics[width=\unitlength,page=8]{shivaem.pdf}}%
    \put(0.87008975,0.05979076){$X_1$}%
    \put(0,0){\includegraphics[width=\unitlength,page=9]{shivaem.pdf}}%
    \put(0.72648069,0.22320809){$Y_1$}%
    \put(0,0){\includegraphics[width=\unitlength,page=10]{shivaem.pdf}}%
    \put(0.72152923,0.07959902){$Y_2$}%
    \put(0,0){\includegraphics[width=\unitlength,page=11]{shivaem.pdf}}%
    \put(0.77104862,0.00531748){{\smash{Graph $G$}}}%
	\end{picture}%

\caption{Sewing graphs by joining neighbors of vertices $A$ and $B$.}
\label{fig:AB-sewing}
\end{center}
\end{figure}

\begin{proof}
  Let $f_1$ and $f_2$ be winning strategies in games $\HG_1$ and $\HG_2$.
  Let us construct a winning strategy for game $\HG$.

  Let $c_1$ be an arbitrary hat arrangement on $N_A$.  To this arranement, we 
  associate a color $g_1(c_1)$ which is the guess of sage $A$ in accordance with
  strategy $f_1$ for this arrangement. Analogously, to each
  arrangement $c_2$ on $N_B$, we associate a color $g_2(c_2)$. All the sages in
  $N_A$ can determine $g_2(c_2)$ and all the sages in $N_B$ can determine $g_1(c_1)$.

  The winning strategy is as follows. The sages on $V(G)\setminus (N_A\cup N_B)$
  use their initial strategies $f_1$ and $f_2$. The sages on $N_A$ also use
  strategy $f_1$ if the hat color of sage $A$ is $g_2(c_2)$. The sages on $N_B$
  use strategy $f_2$ if the hat color of sage $B$ is not $g_1(c_1)$; denote
  this color by $\overline{g_1}(c_1)$  (recall that $h(B)=2$).

  Why this strategy is winning? If $g_1(c_1)=g_2(c_2)$, then on $G_2$-part of
  the new graph we have game  $\HG_2$, where the color of sage $B$ is
  $\overline{g_1}(c_1)=\overline{g_2}(c_2)$, and all the sages use strategy $f_2$ (the guess of $B$ is not determined). But $B$ guesses wrong in game $\HG_2$ with the
  hat arrangement under consideration, and $f_2$ is the winning strategy in game
  $\HG_2$. Therefore someone in $V(G_2)\setminus B$ guesses correctly.

  Analogously, if $g_1(c_1)\neq g_2(c_2)$, we have the game $\HG_1$, where the color
  of $A$ is ${g_2(c_2)}=\overline{g_1}(c_1)$, all sages use strategy $f_1$ and $A$
  says nothing. But $A$ guesses wrong in accordance with strategy $f_1$, and $f_1$ is
  winning strategy in game $\HG_1$. Therefore, someone in $V(G_1)\setminus A$
  guesses correctly.
\end{proof}

The following constructor allows to fasten several graphs $G_i$ by marking 
several vertices in each graph with the names as in graph $G$ and joining them
together in the same way as the corresponding vertices are joined in $G$.

\begin{theorem}
  Let a game $\langle G, h \rangle$ be winning, where $V(G)=\{A_1, A_2, \ldots,
  A_k\}$. Let $G_i$ be a graph $(i=1,\dots, k)$,  $B_{ij}\in V(G_i)$ be set of
  marked vertices in $G_i$  (we assume that sets $V(G_i)$ are disjoint, the
  numbers of marked vertices in different graphs are not necessarily the same, see
  fig.~\ref{fig10}), and let the games $\langle G_i, h_i\rangle$ be winning. Consider a
  new graph $G'=\langle{ V(G'), E(G') }\rangle$, where
  $$
  V(G')=V(G_1)\cup\ldots \cup V(G_k), \quad
  E(G')=E(G_1)\cup\ldots\cup E(G_k)\cup \{B_{i_1j_1}B_{i_2j_2}\mid
  A_{i_1}A_{i_2}\in E(G)\}.
  $$
  Then the game $\langle G', h'\rangle$ is winning, where
  $$
  h'(x)=\begin{cases}
    h_i(x),&  \text{if $x$ belongs to one of the sets  \ }
    V(G_i)\setminus\{B_{i1},B_{i2},\ldots\}, \\
    h_i(x)+h(A_i)-1, & \text{if $x$ coincides with vertex $B_{ij}$}.
  \end{cases}
  $$
\end{theorem}

\begin{figure}[h]%
\setlength{\unitlength}{212bp}\footnotesize
\begin{center}
\footnotesize
 \begin{picture}(1,0.79517469)%
    \put(0,0){\includegraphics[width=\unitlength,page=1]{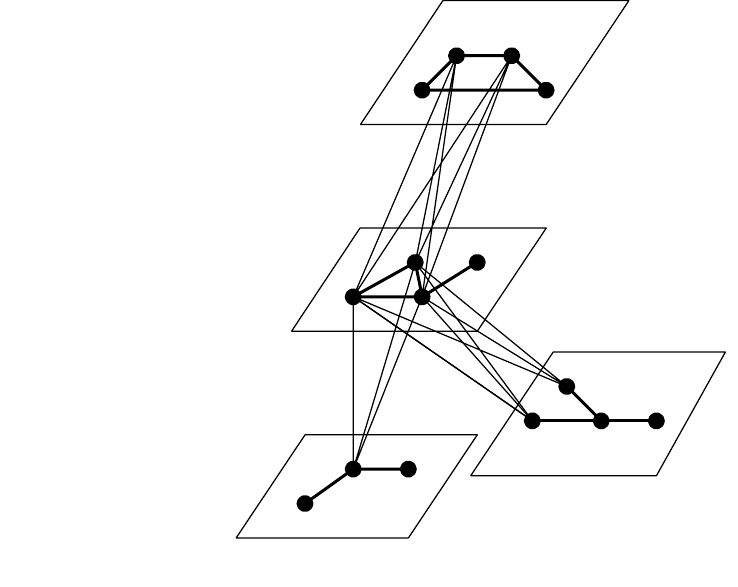}}%
    \put(0.6,0.094){$G_1$}%
    \put(0.94554664,0.21544704){$G_2$}%
    \put(0.693441,0.38351428){$G_3$}%
    \put(0.78681546,0.663628){$G_4$}%
    \put(0.46,0.10340056){$B_{11}$}%
    \put(0.68410505,0.16875751){$B_{21}$}%
    \put(0.78681546,0.27146792){$B_{22}$}%
    \put(0.5907353,0.74766168){$B_{41}$}%
    \put(0.70278178,0.74766168){$B_{42}$}%
    \put(0,0){\includegraphics[width=\unitlength,page=2]{skrepl.pdf}}%
    \put(0.0025,0.01){Graph $G$}%
    \put(0.0025,0.11){$A_1$}%
    \put(0.10520079,0.21544704){$A_2$}%
    \put(0.00249511,0.39285035){$A_3$}%
    \put(0.0025,0.673){$A_4$}%
    \put(0.42,0.36){$B_{31}$}%
    \put(0.475,0.45){$B_{32}$}%
    \put(0.58,0.36){$B_{33}$}%
  \end{picture}%
\caption{Fastening several graphs with the help of graph $G$.}
\label{fig10}
\end{center}

\kern-3mm
\end{figure}

\begin{proof}
  Let  $f$, $f_1$, \dots, $f_k$ be winning strategies for games $\langle
  G,h\rangle$, $\langle G_1,h_1 \rangle$, \dots, $\langle G_k,h_k\rangle$
  respectively. We may assume that sage $B_{ij}$ has hats of $h_i(B_{ij})$ old colors
  and $h(A_{i})-1$ new colors. A \emph{megasage} $M_i$ is a set of sages
  $\{B_{ij}\mid j=1,2,\ldots \}$. Define the color of hat of megasage as a number from 0 to
  $h(A_{i})-1$ as follows. We set it to be 0 if all hat colors of
  sages $B_{i1},B_{i2},\ldots$ are old; otherwise, we set it to be equal the
  maximum new color number of the hats  $B_{i1},B_{i2},\ldots$.

  Now each megasage (i.\,e.\ each sage in this set) understands which hat colors
  his neighboring megasages in graph $G$ have. Thus the megasages can use
  strategy $f$: if a megasage has to say a~new color, let all sages forming this
  magasage say this new color; if megasage $M_i$ has to say 0, then let sages
  $B_{i1},B_{i2},\ldots$  use strategy $f_i$ (looking at the neighbors in the
  component $G_i$ only). Let the sages in
  $V(G_i)\setminus\{B_{i1},B_{i2},\ldots\}$ also use strategy $f_i$. To make our
  strategy well-defined, we append the following rule: if a sage is assigned to play
  strategy $f_i$ but he sees new color on the hats of his neighbors, then his guess
  is not defined by the strategy, and we allow him to say an arbitrary guess.

  Since strategy $f$ is winning, one of the megasages, say $M_{i_0}$, guesses correctly. If
  his color is new then all sages $B_{i_0j}$ $(j=1,2,\ldots)$ say this color and
  one of them certainly guesses correctly. If $M_{i_0}$ has color~0, then sages
  $B_{i_01},B_{i2},\ldots$ use strategy $f_i$
  and the other sages in $V(G_{i_0)})\setminus\{B_{i_01},B_{i_0)2},\ldots\}$
  use strategy $f_i$ too. Thus someone in $V(G_i)$ guesses correctly.
\end{proof}

We note that if only one vertex $B_i$ is marked in each component
$G_i$, then the game on graph $G'$ remains winning even if we greatly
increase values of hatnesses $h(B_i)$. The following lemma holds.

\begin{lemma}
  Let a game $\langle G, h'\rangle$ be winning, $V(G)=\{A_1, \dots, A_n\}$. Take
  $n$ winning games $\langle G_i, h_i\rangle$ such that the sets $V(G_i)$ are
  disjoint, and mark one vertex $A_i$ in each graph $G_i$. Join the marked
  vertices as in graph $G$. Define the hat function in the obtained graph as
  $h(A_i) = h'(A_i)h_i(A_i)$ (all other vertices have the same hatness as in the
  initial graphs). Then the sages win.
\end{lemma}

Lemma follows from theorem \ref{thm:multiplication} on a game product.

The next constructor combines ideas of theorem~ \ref{thm:multiplication} on game
product and theorem~\ref{thm:substitution} on substitution in a single monster.
We glue several winning graphs with common vertex $O$ and create a copy of graph
$G$ by joining neighbors of $O$. For suitable hat function the obtained graph is
winning.

Recall that for $h(A)=mn$, we may consider hat color of $A$ as pair $(c_1,
c_2)$, where $0\leq c_1\leq m-1$, $0\leq c_2\leq n-1$; we say in this case that
color of sage $A$'s hat is composite.

\begin{figure}[h]%
\setlength{\unitlength}{308bp}\footnotesize
\begin{center}
\footnotesize
\begin{picture}(1,0.53847929)%
    \put(0,0){\includegraphics[width=\unitlength,page=1]{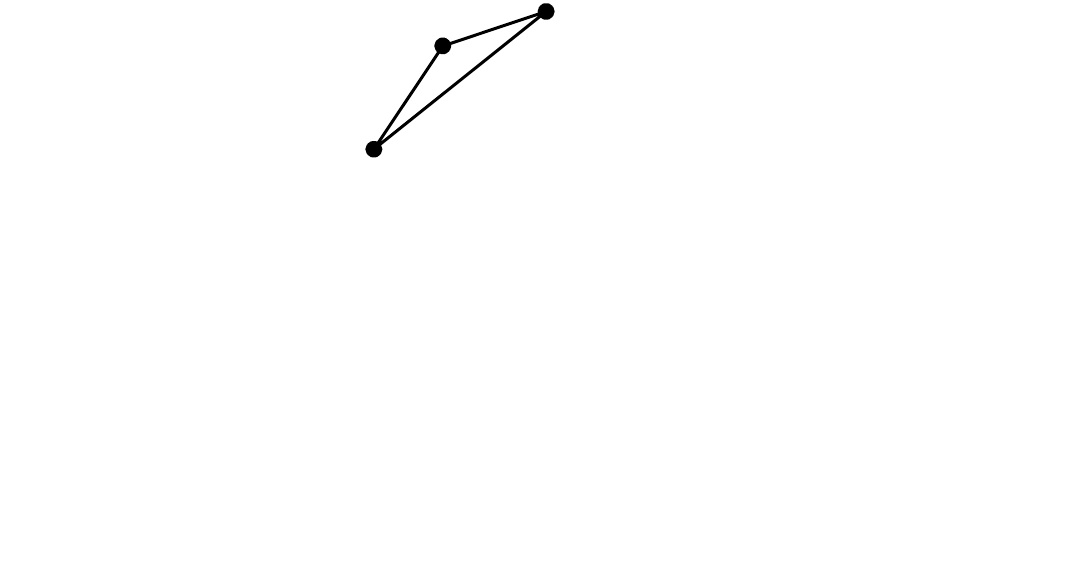}}%
    \put(0.31039658,0.38632374){$O$}%
    \put(0.4,0.52){$A_3$}%
    \put(0,0){\includegraphics[width=\unitlength,page=2]{konus.pdf}}%
    \put(0.31039658,0.2898633){$O$}%
    \put(0.467,0.328){$A_2$}%
    \put(0,0){\includegraphics[width=\unitlength,page=3]{konus.pdf}}%
    \put(0.31039658,0.19339962){$O$}%
    \put(0.41,0.08){$A_1$}%
    \put(0,0){\includegraphics[width=\unitlength,page=4]{konus.pdf}}%
    \put(0.00171847,0.51494098){$A_3$}%
    \put(0.03387197,0.08407585){$A_1$}%
    \put(0.09817893,0.32201678){$A_2$}%
    \put(0.00171847,0){Graph $G$}%
    \put(0.3232599,0){Graphs $G_1$, $G_2$, $G_3$}%
    \put(0.83,0){``Cone''}%
    \put(0,0){\includegraphics[width=\unitlength,page=5]{konus.pdf}}%
    \put(0.72840163,0.2898633){$O$}%
    \put(0.84,0.17569433){$A_1$}%
    \put(0,0){\includegraphics[width=\unitlength,page=6]{konus.pdf}}%
    \put(0.88274227,0.32844678){$A_2$}%
    \put(0,0){\includegraphics[width=\unitlength,page=7]{konus.pdf}}%
    \put(0.80557191,0.41848047){$A_3$}%
  \end{picture}%
\caption{``Cone'' over graph $G$.}\label{fig:konus}
\end{center}
\end{figure}

\begin{theorem}[On a ``cone'' with vertex $O$ over graph $G$]\label{thm:konus}
  Let a game $\HG=\langle G, h \rangle$, where $V(G)=\{A_1, A_2, \ldots, A_k\}$
  and $k$ games $\HG_i=\langle G_i, h_i\rangle$, $1\leq i\leq k$, be winning, and the
  sets $V(G_i)$ be disjoint. Assume that in each $G_i$ one vertex is labeled $O$
  so that $h_1(O)=h_2(O)=\ldots=h_k(O)$, and also one of the neighbors $A_i$ of $O$ is labeled. Let  $G'=\langle{V(G'),  E(G')}\rangle$ be a new graph, 
  \begin{align*}
    V(G')&=(V(G_1)\setminus\{O\})\cup\ldots \cup (V(G_k)\setminus\{O\})\cup\{O\},
    \\
    E(G')&=E(G_1)\cup\ldots\cup E(G_k)\cup \{A_i A_j\mid A_{i}A_{j}\in E(G)\}.
  \end{align*}
  Then the game $\langle G', h'\rangle$ is winning, where
  $$
  h'(x)=\begin{cases}
    h_i(x),         & \text{if $x$ belongs to one of sets  \ } V(G_i)\setminus\{A_i\}, \\
    h_i(A_i)h(A_i), & \text{if $x$ coincides with $A_i$}.
  \end{cases}
  $$
\end{theorem}

\begin{proof}
  Let us describe a winning strategy. The sages in vertices $A_i$ have composite
  colors, let they play two strategies simultaneously: strategy of game $\HG_i$
  for the first coordinate of the color and strategy of $\HG$ for the second
  one. The sages in  $V(G_i)\setminus\{O, A_i\}$ use the strategy of game
  $\HG_i$ (neighbors of~$A_i$ pay attention only to $G_i$-coordinate of his color).

 Tthe most cunning role goes to sage $O$. He sees all the sages $A_i$ and knows, which
  sage guesses the $G$-coordinate of his color correctly, let this sage be
  $A_j$. Then $O$ looks only at his neighbors in graph $G_j$ and use strategy
  $h_j$ (he looks at the $G_j$-coordinate of $A_j$'s hat color only).

  As a result, someone on graph $G_j$ guesses correctly (if this sage is $A_j$, then
  he guesses both components of his composite color).
\end{proof}

An \,e\,x\,a\,m\,p\,l\,e \,of winning graph obtained by theorem is shown in
fig.~\ref{fig:primer-konus}.
Here, graph $G$ is a complete graph on 3 vertices with hatnesses 3, 3, 3; $G_1$ is a
5-cycle with hatnesses  4, 2, 3, 3, 3; the graphs $G_2$ and $G_3$ are 4-cycles with
hatnesses 4, 2, 3, 3. The latter three graphs are winning by
corollary~\ref{cor:cycle_233}. $O$~is vertex of hatness 4 in these cycles.

\begin{figure}[h]%
\setlength{\unitlength}{140bp}\footnotesize
\begin{center}
\footnotesize
\begin{picture}(1,1.0567848)%
    \put(0,0){\includegraphics[width=\unitlength,page=1]{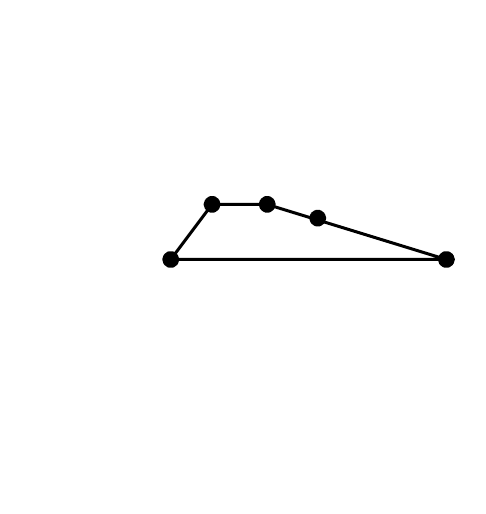}}%
    \put(0.35173995,0.45135026){$O$}%
    \put(0.33754971,0.56488024){4}%
    \put(0.91942565,0.45135026){$A_1$}%
    \put(0.91942565,0.55069){6}%
    \put(0.66,0.53){3}%
		\put(0.4,0.67){3}%
		\put(0.5,0.67){3}%
    \put(0,0){\includegraphics[width=\unitlength,page=2]{primer7.pdf}}%
    \put(0.09628492,1.00483856){$A_2$}%
    \put(-0.0030618,0.9480699){6}%
    \put(0.1,0.65){3}%
    \put(0.15,0.55){3}%
    \put(0,0){\includegraphics[width=\unitlength,page=3]{primer7.pdf}}%
    \put(0.09628492,0.35200354){3}%
    \put(0.15,0.45){3}%
    \put(0.09,0.01){$A_3$}%
    \put(-0.0030618,0.03977278){6}%
    \put(0,0){\includegraphics[width=\unitlength,page=4]{primer7.pdf}}%
  \end{picture}%
\caption{The sages win by theorem \ref{thm:konus}.}\label{fig:primer-konus}
\end{center}
\end{figure}

\medbreak
In the proof of theorem \ref{thm:konus}, sage $O$ plays the role of ``dispatcher'':
looking at sages $A_i$, he chooses in which component the winning game is played.
The success of this choice is provided by his knowledge of the winner in the
graph $G$. Now we consider a case in which sage $O$ cannot see the entire
graph $G$.

Let us fix a graph $G$, a hat function and an arbitrary strategy of sages on $G$. 
A set $S\subset V(G)$ is said to be \emph{predictable} if it satisfies the property: for
any hat arrangement on graph $G$, one can choose a sage $A\in S$ looking at the hat colors in the set $S$ only, and so that if for this hat arrangement some sages in
$S$ guess their colors correctly, then $A$ is one of these ``winners''.

An example of predictable set $S$ is a 5-clique in the  graph ``Big bow'' for the
strategy from theorem~\ref{thm:bantik} (fig.~\ref{fig:big-bow}). Indeed, the
strategy in the proof of theorem ~\ref{thm:bantik} for the sages on any of 5-cliques
consists of checking some hypotheses about the sums of colors over vertices of the
clique. Anyone who sees the hat arrangement only on the clique can
determine who of the non-central sages guesses correctly. If nobody guesses, then
the only person who can guess correctly in this set of vertices is the central
sage.

A simpler examples of predictable set is the component $G_1$ or $G_2$ in
theorem~\ref{thm:multiplication} on graph products (for the strategies from the
proof). For example, graph $G$ depicted in  fig.~\ref{fig:konus-monstr}, left,
with all hatnesses equal 4, is the product of two 3-cliques $S_1$ and $S_2$ with
hatnesses 2, 4, 4. Hence, sets $S_1$ and $S_2$ are predictable.

Returning to ``cone'' theorem we note that the construction of a new graph can be
generalized to the case of several dispatchers, that collectively see the entire
graph $G$ but individually each dispatcher sees some predictable part of the
graph $G$ only. Let us describe this generalization more precisely.

Let $\HG=\langle G, h \rangle$ be a winning game, where $V(G)=\{A_1, A_2, \ldots,
A_k\}$. Fix a winning strategy~$f$ in this game. Let the vertex set of $G$ be
a union of several predictable sets with respect to~$f$:
$V(G)=S_1\cup S_2\cup\ldots\cup S_m$. Sets  $S_j$ can intersect; if a
vertex $A_\ell$ belongs to several~$S_j$, we consider one of them (any) as
``principal'' for $A_\ell$, the number of this set is denoted by $j_\ell$. For instance, one can take
$j_\ell=\min\{j: A_\ell\in S_j\}$. Assume that none of~$S_j$ is a subset of the
union of the other sets. Let for each~$\ell$, $1\leq \ell\leq k$, a winning game
$\HG_\ell=\langle G_{\ell}, h_{\ell}\rangle$ is given (the sets $V(G_\ell)$ are
pairwise disjoint), in each graph $G_{\ell}$ one arbitrary vertex
$A_\ell$ is marked and one more vertex, neighboring to~$A_\ell$, is labeled
as~$O_{j_\ell}$. Thus, some vertices in different graphs can be labeled with the same
label $O_j$. We assume that for each $j$, $1\leq j\leq m$, the hatnesses of all
vertices~$O_j$ are equal, denote this value by~$o_j$.

Consider a graph $G'=(V(G'), E(G'))$, where $V(G')=\bigcup\limits_{\ell=1}^{k}
V(G_{\ell})$. We assume that the vertices with the same labels (i.\,e.\ different
copies of vertices $O_j$) are identified in this union,
$$
E(G')=E(G_1)\cup\ldots\cup E(G_k)\cup E(G)\cup \{O_j A\mid 1\leq j\leq m, \ A\in S_j\}.
$$
The latter set in the union provides the ability of sages $O_j$ to see all
vertices of set $S_j$ including those vertices, for which index $j$ is not principal.
We consider the union formally: if a set $E(G_\ell)$ contains edge  $A_\ell
O_{j_\ell}$, then $E(G')$ is also contains edge $A_\ell O_{j_\ell}$, since there are
vertices denoted by  $A_\ell$ and $O_{j_\ell}$ in graph $G'$.

An example is given in fig.~~\ref{fig:konus-monstr}. Here we have $j_1=j_2=j_3=1$, $j_4=j_5=2$. Graph $G$ is a ``small bow'', we have checked above that sets $S_1$ and $S_2$ are predictable.

Define hat function  $h'$ on graph $G'$
$$
h'(x)=\begin{cases}
  h_{\ell}(x),     & \text{if $x$ belongs to one of the sets \ } V(G_{\ell})\setminus\{O_{j_\ell}, A_\ell\},\\
  o_j,             & \text{if $x$ coincides with $O_j$}, \\
  h_{\ell}(A_\ell)h(A_\ell), & \text{if $x$ coincides with  $A_\ell$}.
\end{cases}
$$

\begin{figure}[h]%
\setlength{\unitlength}{440bp}\footnotesize
\begin{center}
\begin{picture}(1,0.38490514)%
    \put(0,0){\includegraphics[width=\unitlength,page=1]{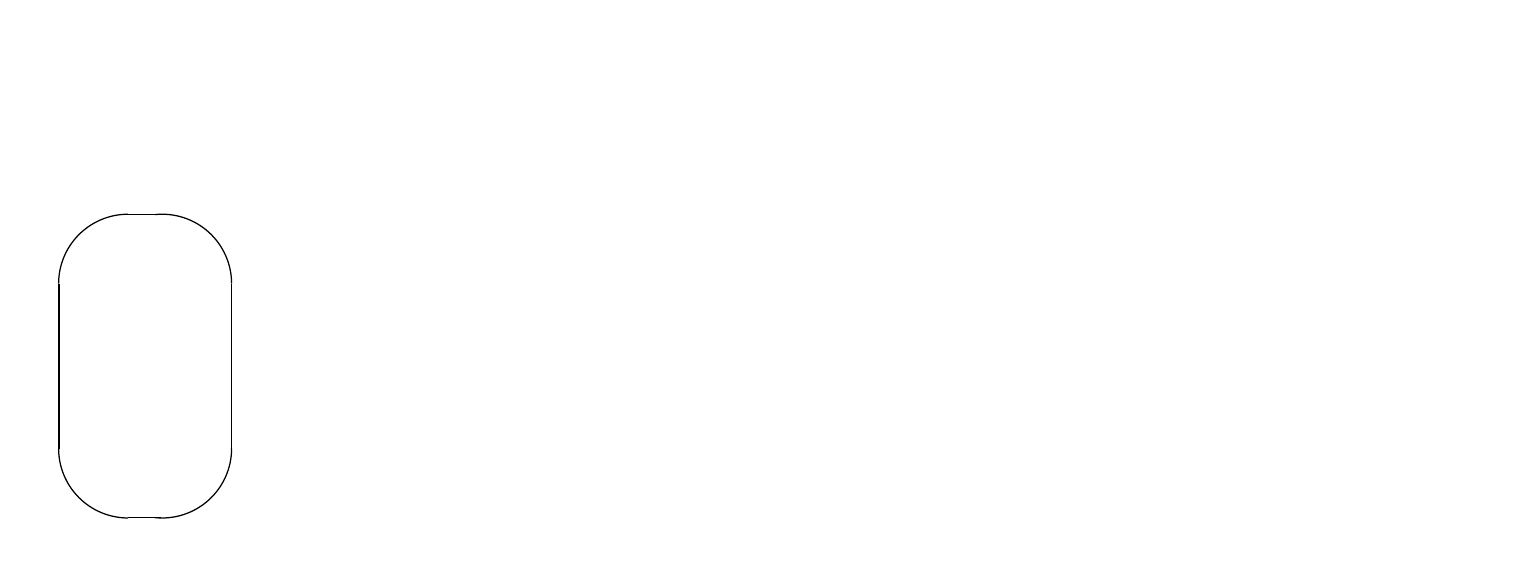}}%
    \put(0.00889269,0.10413763){$S_2$}%
    \put(0,0){\includegraphics[width=\unitlength,page=2]{dva-disp.pdf}}%
    \put(-0.00013278,0.28466949){$S_1$}%
    \put(0,0){\includegraphics[width=\unitlength,page=3]{dva-disp.pdf}}%
    \put(0.08561811,0.20343124){$A_1$}%
    \put(0.05402675,0.36139488){$A_3$}%
    \put(0.07659266,0.05900584){$A_5$}%
    \put(0.12172447,0.27112903){$A_2$}%
    \put(0.12172447,0.13573134){$A_4$}%
    \put(0.05402675,0.00484634){Graph $G$}%
    \put(0,0){\includegraphics[width=\unitlength,page=4]{dva-disp.pdf}}%
    \put(0.23811365,0.27112903){$O_1$}%
    \put(0.29678592,0.36139488){$A_3$}%
    \put(0,0){\includegraphics[width=\unitlength,page=5]{dva-disp.pdf}}%
    \put(0.23811365,0.20343124){$O_1$}%
    \put(0.34643273,0.22599719){$A_2$}%
    \put(0,0){\includegraphics[width=\unitlength,page=6]{dva-disp.pdf}}%
    \put(0.23811365,0.13573134){$O_1$}%
    \put(0.3148391,0.05900584){$A_1$}%
    \put(0.25165411,0.00484634){Graphs $G_1$, $G_2$, $G_3$}%
    \put(0,0){\includegraphics[width=\unitlength,page=7]{dva-disp.pdf}}%
    \put(0.52149419,0.20343124){$O_2$}%
    \put(0.62981324,0.22599719){$A_4$}%
    \put(0,0){\includegraphics[width=\unitlength,page=8]{dva-disp.pdf}}%
    \put(0.52149419,0.13573134){$O_2$}%
    \put(0.59821958,0.05900584){$A_5$}%
    \put(0.55308779,0.00484634){Graphs $G_4$, $G_5$}%
    \put(0,0){\includegraphics[width=\unitlength,page=9]{dva-disp.pdf}}%
    \put(0.80487466,0.27112903){$O_1$}%
    \put(0.85000651,0.20794404){$A_1$}%
    \put(0,0){\includegraphics[width=\unitlength,page=10]{dva-disp.pdf}}%
    \put(0.91770641,0.3){$A_2$}%
    \put(0,0){\includegraphics[width=\unitlength,page=11]{dva-disp.pdf}}%
    \put(0.86354691,0.36139488){$A_3$}%
    \put(0,0){\includegraphics[width=\unitlength,page=12]{dva-disp.pdf}}%
    \put(0.80487466,0.13573134){$O_2$}%
    \put(0.85,0.07254635){$A_5$}%
    \put(0,0){\includegraphics[width=\unitlength,page=13]{dva-disp.pdf}}%
    \put(0.91770641,0.16280993){$A_4$}%
    \put(0,0){\includegraphics[width=\unitlength,page=14]{dva-disp.pdf}}%
    \put(0.85903194,0.00484634){Graph $G'$}%
  \end{picture}%
	\caption{``Cone'' over a graph $G$ with two dispatchers.}\label{fig:konus-monstr}
\end{center}
\end{figure}

\begin{corollary}
  In the above notations, the game $\langle G', h'\rangle$ is winning.
\end{corollary}

\begin{proof}
  Each sage $A_\ell$ has a composite color and uses two strategies: the strategy
  of game $\HG_{\ell}$ for the first component of the color, and the strategy of
  game $\HG$ for the second one. The sages from $V(G_{\ell})\setminus\{O_{j_\ell},
  A_\ell\}$ use the strategy of game $\HG_{\ell}$ (the neighbors of $A_\ell$
  see at the  $G_\ell$-cordinate of his color only).

  Each of the sages $O_j$ sees the predictable component $S_j$ of graph
  $G$ and hence knows which sage in this component (if exists) guesses
  the $G$-coordinate of his own color correctly. Let sage $O_j$ knows that sage
  $A_\ell$ guesses correctly. If $j= j_\ell$, then $O_j$ sees only at his
  neighbors on subgraph $G_\ell$ and uses strategy $h_\ell$ (taking into account
  the $G_\ell$-coordinate of sage $A_\ell$'s color only). If $j\ne j_\ell$ or if
  none of sages in  $S_j$ guess correctly, $O_j$ can choose an arbitrary guess.

  Now let $\ell$ be an index, for which sage $A_\ell$ guesses the $G$-coordinate of his
  color correctly. Then on graph $G_\ell$ either one of the sages (not $A_\ell$)
  guesses correctly, or $A_\ell$ guesses correctly both coordinates of his color.
\end{proof}

\section{Blind chess}

In this section we present a new game which is in fact a special case of \hats
game on 4-cycle. This game gives us a whole class of new games on 
cooperative guessing. All you need to change in the initial \hats game is the
target of guessing. Here we replace the guessing of marked element in the set
(i.\,e.\ a color of hat) with making a check to invisible king! In general, the
sages can try to perform any actions,  for which in the absence of information 
100\,\% success is not guaranteed.

\subsection{Rook check}

\begin{definition}
  The game ``Rook check''.
  Two chess players $\mathcal L$ and $\mathcal R$ are sitting opposite each
  other and there is a chessboard on the wall behind each of them. Each chess
  player does not see his own board (which is behind him) but sees the board of the other chess
  player. The referee places one black king on each of these boards. So the players
  see the king on another board but do not see the king in their own board.
  After that, each chess player, independently of the other, point to one square of his own chessboard
   and the referee puts a white rook on this square. If
  at least one of the kings is under attack by the rook (or the
  rook is placed on the square where the king is), then the chess players
  both win, otherwise they lose.
\end{definition}

Chessboards of the players can be different and have arbitrary sizes, which are
known to the players. As in the \hats game, the chess players determine public
deterministic strategy in advance. The referee knows this strategy and plays
against the chess players.

Let us explain how  Rook check game relates to the \hats game. Let a graph
$G$ be the 4-cycle $ABCDA$ with hat function $h$. In fact, graph~$G$ is a
complete bipartite graph~$K_{2,2}$ with parts $\{A, C\}$ and $\{B, D\}$. The pair
of players $A$ and $C$ is called a chess player~$\mathcal L$, his board has
size $h(A)\times h(C)$. The pair $B$ and $D$ is called a chess player $\mathcal
R$, his board has size $h(B)\times h(D)$.

The hat colors of $A$ and $C$ can be interpreted as coordinates of the cell where
the king is placed. Since $A$ and $C$ do not see each other, they know nothing
about king placement on their board. The pair of colors that $A$ and $C$ say
can be interpreted as a \emph{cross} on the chessboard, i.\,e.\ a configuration 
consisting of one horizontal and one vertical line, or which is the same, a
position for chess rook. It is clear that one or both chess players guess their
colors if and only if the king is under attack of the rook. Similar interpretations are valid
for $B$ and $D$.

Thus, the \hats game on cycle $ABCDA$ with hat function $h$ is equivalent to the game
Rook check on the boards $L(h(A)\times h(C))$ and~$R(h(B)\times h(D))$. It is
clear that the result of the game does not depend on, which board is the left
and which is the right.

Generally, we can define the Rook check game in the case where $n$ chess players
sit in the vertices of an arbitrary graph: each player has his own chessboard but
sees only the boards of his neighbors (and does not see his own chessboard). The
aim of the players is similar, they want at least one of kings to be under
attack. This game is equivalent to \hats game  on a ``doubled'' graph. We will
not discuss this game here.

\medbreak

Let us return to the game of two players on the boards $L(a\times c)$ and $R(b\times
d)$. We use the following notations.

Let us number the cells of $L(a\times c)$ board from left to right from top to bottom, see
fig.~\ref{fig:2x3and3x4}{\it a}, where we use boards $L(2\times 3)$ and $R(3\times 4)$
as examples. Let the strategy of chess player~$\mathcal R$ be given by the table
as in fig.~\ref{fig:2x3and3x4}{\it b}. Here we put $ac$ labels  $r_i$ in the
cells of $R(b\times d)$  board (a~cell can contain several labels $r_i$), where the index $i$
runs over all numbers of the cells of $L(a\times c)$ board.
The label~$r_i$ means that chess player~$\mathcal R$, seeing his partner's king
is on the $i$-th cell of $L(a\times c)$ board, puts his rook on the
cell of $R(b\times d)$ board with the label~$r_i$.

The strategy of chess player $\mathcal L$ is also fiven with help of
$R(b\times d)$ board, see fig.~\ref{fig:2x3and3x4}{\it c}. Here there is a
number from 1 to $ac$ in each cell of $R(b\times d)$, the numbers denote
cells of $L(a\times c)$ board. Each cell of $R(b\times d)$ board contains
exactly one number, some numbers from 1 to $ac$ can be absent in this table and some
numbers can repeat. When $\mathcal L$ sees that $\mathcal R$'s king is located
on $R(3\times 4)$  board in the cell labeled~$k$,  he puts the rook on
$k$-th cell of $L(a\times c)$ board.

To avoid misunderstandings in notations, we use labels of type ``letter~$r$
with index'' for chess player~$\mathcal R$, and labels of type ``number'' for
chess player~$\mathcal L$. The lines on the board~$L$ are called rows and columns, whereas
the lines on the board $R$ are called verticals and horizontals.

\begin{figure}[h]
\setlength{\unitlength}{11pt}
\footnotesize
\leavevmode\hfil
\vbox{\hsize=3.3cm
\begin{tabular}{|c|c|c|}
\hline
1&2&3\\
\hline
4&5&6\\
\hline
\end{tabular}
\par
\vskip5pt
a) Cell numbers on $L$ board
}
\hfill
\begin{picture}(4,4)(0.8,-1)
\multiput(0,0)(0,1){4}{\line(1,0){4}}
\multiput(0,0)(1,0){5}{\line(0,1){3}}
\put(0.1,2.3){$r_4$}
\put(2.1,0.3){$r_2$}
\put(3.1,1.3){$r_3$}
\put(3.1,0.3){$r_1$}
\put(1.1,2.3){$r_5$}
\put(0.1,1.3){$r_6$}
\put(0.05,2.05){\shmt Z}
\put(1.05,1.05){\shmt Z}
\put(-3.7,-1.2){b) Strategy of player $\mathcal R$}
\end{picture}
\hfill
\begin{picture}(8,3)(0,-1)
\multiput(0,0)(0,1){4}{\line(1,0){4}}
\multiput(0,0)(1,0){5}{\line(0,1){3}}
\put(0.2,2.2){1}\put(1.2,2.2){3}\put(2.2,2.2){3}\put(3.2,2.2){5}
\put(0.2,1.2){2}\put(1.2,1.2){1}\put(2.2,1.2){4}\put(3.2,1.2){5}
\put(0.2,0.2){2}\put(1.2,0.2){6}\put(2.2,0.2){6}\put(3.2,0.2){4}
\put(-3.7,-1.2){c) Strategy of player $\mathcal L$}
\end{picture}
\caption{Notations for strategies.}
\label{fig:2x3and3x4}%
\end{figure}

\begin{definition}
  Let the king be in the $i$-th  cell of $L(a\times c)$ board.  A cell of
  $L(a\times c)$ board is said to be \emph{$i$-weak}  if the rook does not attack the king from
  this cell.  For example, cells 5 and 6 on the $L(2\times 3)$ board
  (fig.~\ref{fig:2x3and3x4}{\it а}) are 1-weak.
\end{definition}

\begin{lemma}\label{lem:chess-strategy}
  Let $L(a\times c)$ and $R(b\times d)$ be the boards in the game Rook check.
  A strategy is winning if and only if for each $i$, $1\leq i\leq ac$, all cells on
  $R(b\times d)$ board labeled  with numbers of $i$-weak cells belong to the cross
  with center $r_i$.
\end{lemma}

\begin{proof}
  Let cell $\ell$ of $L(a\times c)$ board be $i$-weak and the referee put the kings on the
  cell $i$ of $L(a\times c)$ board and the cell of $R(c\times d)$
  board labeled by $\ell$. Then player~$\mathcal L$ according to his strategy puts the rook on the
  cell $\ell$ of $L(a\times c)$ board, and it does not attack the king. In the same
  time, player~$\mathcal R$ puts his rook on the cell of $R(c\times d)$ board
  labeled by $r_i$. The players win if and only if this rook attacks the king, i.\,e.\ the cell
  labeled by~$\ell$ is in the cross with center $r_i$.
\end{proof}

This lemma provides the following property of winning strategies: if the cell $\ell$
on board $L(a\times c)$ is simultaneously $i$-weak, $j$-weak, etc., then all the
cells on $R(b\times d)$ board labeled by $\ell$ (if they exist) are located in the
intersection of the crosses with centers $r_i$, $r_j$ etc. For example, for the
strategies in fig.~\ref{fig:2x3and3x4} (we will prove below that they are
winning) both cells labeled by 1 on $R(3\times 4)$ board belong to intersection of
crosses  $r_5$ and $r_6$ (the shaded area in fig.~\ref{fig:2x3and3x4}{\it b})
because the cells 5 and 6 on board $L(2\times 3)$ are 1-weak.

\smallbreak
The following theorem gives a complete analysis of the game Rook check for two players.
We assume that the number of horizontals of each board does not exceed the number of verticals and that the left board has the shortest vertical size.

\begin{theorem}\label{thm:Rook_chess}
  The chess players win in the game Rook check on the following boards:
  \begin{enumerate}[label=Win\arabic*),leftmargin=*,labelindent=\parindent,nosep]
  \item if one of the boards has sizes $1\times k$, where $k$ is an arbitrary
    positive integer;
  \item $L(2\times k)$ and $R(2\times m)$,  where $k$ and $m$ are arbitrary
    positive integers;
  \item\label{itm:win2333} $L(3\times 3$), $R(3\times 3)$;
  \item $L(2\times 3)$, $R(3\times 4)$;
  \item $L(2\times 4)$, $R(3\times 3)$;
  \item $L(2\times 2)$, $R(k\times m)$, where $\min(k,m)\leq 4$.
  \end{enumerate}

  The chess players lose on the following boards:

  \begin{enumerate}[label=Lose\arabic*),leftmargin=*,labelindent=\parindent,nosep]
  \item\label{itm:lose2344} $L(2\times 3)$, $R(4\times 4)$;
  \item $L(2\times 3)$, $R(3\times 5)$;
  \item\label{itm:lose2434} $L(2\times 4)$, $R(3\times 4)$;
  \item $L(2\times 5)$, $R(3\times 3)$;
  \item $L(3\times 3)$, $R(3\times 4)$;
  \item $L(2\times 2)$, $R(5\times 5)$.
  \end{enumerate}
\end{theorem}

For  boards of other sizes the question if the sages win can be answered by
comparing with these cases. For example, the chess players lose on the boards
$L(3\times 4)$, $R(3\times 4)$ because they lose even in ``smaller'' case~\ref{itm:lose2434}.
The chess players win on the boards $L(2\times 3)$, $R(3\times 3)$ because they
win even on larger boards (as in case~\ref{itm:win2333}).

\begin{proof}[Proof of the theorem]\mbox{}\\*
  \smallbreak{\bf Win1)} This statement is trivial.
  
  \smallbreak
  {\bf Win2)} In the \hats language the hat function of two neighbor sages in
  the corresponding 4-cycle equals 2, these sages provide a win, even
  not looking at the others.

  \smallbreak
  {\bf Win3)}  This is a retelling to the language of the Rook check game 
  the known statement that the sages win on 4-cycle, if they all obtain hats of
  three colors (\cite{Gadouleau2015, cycle_hats}). For example, the strategy
  of the sages, described in~\cite{spb_school_problems_2016}, looks in Rook check
  language as follows. If a chess player sees that the king of his
  pertner is in the central cell of the board, then he puts his rook on the center too.
  Otherwise he puts the rook on the cell, where the arrow 
  leading from partner's king shows (on the auxiliary diagram for this chess
  player, see fig.~\ref{fig:two3x3str}. The coordinates of cells in the figure
  correspond to the numbers of hat colors. Thus, chess player~$\mathcal L$,
  seeing that the king of his partner is located in the cell~$(2,2)$, puts his rook
  on the cell~$(1,0)$ (this case corresponds to the bold arrow in 
  fig.~\ref{fig:two3x3str} on the left).

\begin{figure}[h]
\setlength{\unitlength}{15pt}
\footnotesize
\leavevmode\hfil
\begin{picture}(3,5)(0,-2)
\multiput(0,0)(0,1){4}{\line(1,0){3}}
\multiput(0,0)(1,0){4}{\line(0,1){3}}

\put(2.5,1.5){\vector(0,1){1}}
\put(2.5,2.5){\vector(-2,-1){2}}
\put(0.5,1.5){\vector(0,-1){1}}
\put(1.5,0.5){\vector(-1,2){1}}
\put(2.5,0.5){\vector(-1,0){1}}
\put(1.5,2.5){\vector(1,-2){1}}
\put(0.5,2.5){\vector(1,0){1}}
\put(1.3,1.3){$\bullet$}
\put(-0.6,0.3){2}\put(-0.6,1.3){0}\put(-0.6,2.3){1}
\put(0.3,-0.7){2}\put(1.3,-0.7){0}\put(2.3,-0.7){1}
\put(-.5,-2){Strategy $\mathcal L$}
\thicklines
\put(0.5,0.5){\vector(2,1){2}}
\end{picture}
\hfil
\begin{picture}(3,3)(0,-2)
\multiput(0,0)(0,1){4}{\line(1,0){3}}
\multiput(0,0)(1,0){4}{\line(0,1){3}}
\put(0.5,0.5){\vector(0,1){1}}
\put(2.5,1.5){\vector(-2,-1){2}}
\put(2.5,2.5){\vector(0,-1){1}}
\put(0.5,1.5){\vector(2,1){2}}
\put(1.5,0.5){\vector(1,0){1}}
\put(2.5,0.5){\vector(-1,2){1}}
\put(1.5,2.5){\vector(-1,0){1}}
\put(0.5,2.5){\vector(1,-2){1}}
\put(1.3,1.3){$\bullet$}
\put(-0.6,0.3){2}\put(-0.6,1.3){0}\put(-0.6,2.3){1}
\put(0.3,-0.7){2}\put(1.3,-0.7){0}\put(2.3,-0.7){1}
\put(-.5,-2){Strategy $\mathcal R$}
\end{picture}

\caption{Four sages stand around non-transparent baobab... }
\label{fig:two3x3str}%
\end{figure}
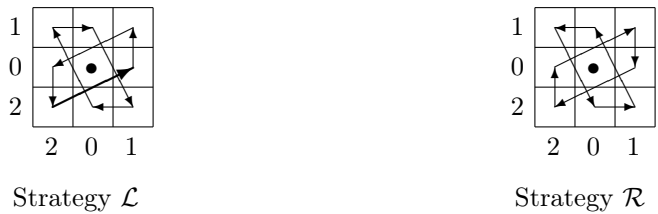

  \smallbreak
  {\bf Win4), Win5)}
  The strategies presented in fig.~\ref{fig:2x3and3x4} and~\ref{fig:2x4and3x3}
  satisfy Lemma~\ref{lem:chess-strategy} (direct check). So the chess players win.

\begin{figure}[h]
\setlength{\unitlength}{18pt}
\footnotesize
\leavevmode
\vbox{\hsize=3.3cm
\begin{tabular}{|c|c|c|c|}
\hline
1&2&3&4\\
\hline
5&6&7&8\\
\hline
\end{tabular}
\par
\vskip5pt
a) Cell numbers on board $L$
}
\hfill
\begin{picture}(3,4)(2,-1)
\multiput(0,0)(0,1){4}{\line(1,0){3}}
\multiput(0,0)(1,0){4}{\line(0,1){3}}
\put(1.3,0.3){$r_4$}
\put(1.5,1.6){$r_2$}
\put(1.3,1.2){$r_3$}
\put(1.05,1.6){$r_1$}
\put(0.3,2.3){$r_5$}
\put(2.3,2.3){$r_6$}
\put(2.3,0.3){$r_7$}
\put(0.3,0.3){$r_8$}
\put(-1.7,-1.){b) Strategy of player $\mathcal R$}
\end{picture}
\hfill
\begin{picture}(5,3)(0,-1)
\multiput(0,0)(0,1){4}{\line(1,0){3}}
\multiput(0,0)(1,0){4}{\line(0,1){3}}
\put(0.3,2.3){3}\put(1.3,2.3){5}\put(2.3,2.3){4}
\put(0.3,1.3){8}\put(1.3,1.3){6}\put(2.3,1.3){8}
\put(0.3,0.3){2}\put(1.3,0.3){7}\put(2.3,0.3){1}
\put(-2.5,-1.){c) Strategy of player $\mathcal L$}
\end{picture}
\caption{Winning strategy for game on boards
$L(2\times 4)$, $R(3\times 3)$.}%
\label{fig:2x4and3x3}%
\end{figure}
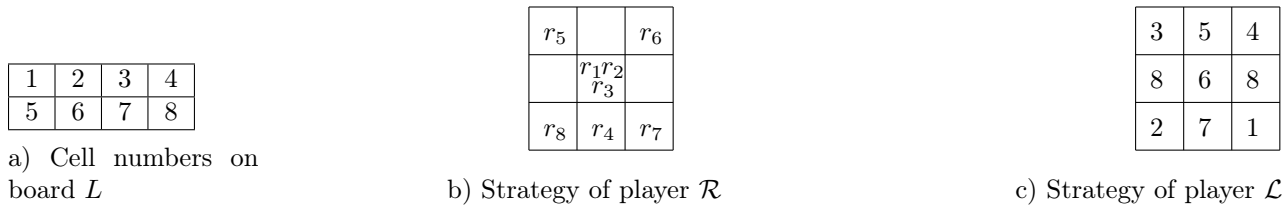

  \smallbreak
  {\bf Win6)}  In \hats language this case means that the 4-cycle contains
  path~$P_3$ with hat function  2, $x$, 2, where $x\leq 4$. The sages win on
  such path by corollary~\ref{cor:path2442}.

  \smallbreak
  {\bf Lose1)}
  We show that the players have no winning strategy in this case.

  Fix a strategy of chess player~$\mathcal R$,  see, for instance,
  fig~\ref{fig:2x3and4x4}{\it b}. Let us try to understand how the strategy of $\mathcal
  L$ looks like, namely, where can the cells with labels 1, 2, and 3 be located on
  $R(4\times 4)$ board. By lemma~\ref{lem:chess-strategy}, the cells with label~1
  belong to the intersection of crosses $r_5$ and $r_6$, the cells with label~2 to
  the intersection of crosses $r_4$ and~$r_6$, and the cells with label~3 to the
  intersection of $r_4$ and $r_5$.

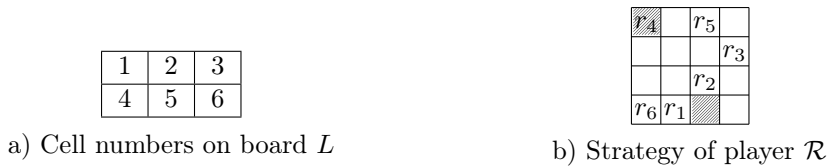
\begin{figure}[h]
\setlength{\unitlength}{11pt}
\footnotesize
\leavevmode\hfil
\vbox{\hsize=4.5cm\centering
\begin{tabular}{|c|c|c|}
\hline
1&2&3\\
\hline
4&5&6\\
\hline
\end{tabular}
\par
\vskip5pt
a) Cell numbers on board $L$
}
\hfil
\begin{picture}(6,5)(0,-1)
\multiput(0,0)(0,1){5}{\line(1,0){4}}
\multiput(0,0)(1,0){5}{\line(0,1){4}}
\put(0.05,3.05){\shmt Z}
\put(2.05,0.05){\shmt Z}
\put(0.1,3.3){$r_4$}
\put(2.1,1.3){$r_2$}
\put(3.1,2.3){$r_3$}
\put(1.1,0.3){$r_1$}
\put(2.1,3.3){$r_5$}
\put(0.1,0.3){$r_6$}
\put(-2.7,-1.2){b) Strategy of player $\mathcal R$}
\end{picture}
\caption{It happens that this strategy is losing.}%
\label{fig:2x3and4x4}%
\end{figure}

Note that the union of pairwise intersection of any three crosses (possibly, coinciding) on $R(4\times 4)$ board contains at most 8 cells. Indeed, let us examine the cases.

1.  If the centers of the crosses belong to different verticals and horizontals, then each pairwise intersection consists of two cells (in the example in  fig.~\ref{fig:2x3and4x4}{\it b}, the intersection of the crosses $r_5$ and $r_6$ is shaided), so we have at most 6 cells totally.

2. If the centers of any two crosses do not coincide and two centers belong to the same horizontal or vertical (as the crosses $r_4$ and $r_5$ in fig.~\ref{fig:2x3and4x4}{\it b}), then the intersection of these two crosses contains 4 cells and adding of the third cross (say, $r_6$) can give 4 more cells to the union of pairwise intersections only if the center of the third cross and one of the first two centers are on the same line (as $r_4$ and $r_6$ in fig.~\ref{fig:2x3and4x4}{\it b}. In this case, we have 8 cells, and 7 of them belong to one cross (cross~$r_4$ in our example).

3. If the centers of some two crosses coincide, then intersection of these crosses contains 7 cells. For any location of the third center the set of pairwise intersections does not increase.

Thus, for the cells with labels 1, 2, 3 on board $R(4\times 4)$ there are at most 8 positions, and similarly, for the cells with labels 4, 5, 6 there are at most 8 positions too. Since $R(4\times 4)$ board contains 16~cells, we have 8~positions for labels 1, 2, 3 and 8~positions for labels 4, 5, 6. But as was established by trying all possible cases 1--3, 8~positions can be realized only as a set ``whole cross plus one cell''. It remains to observe that it is impossible to cover $R(4\times 4)$ board completely by two crosses and two additional cells.

\smallbreak
{\bf Lose2)}  As in \ref{itm:lose2344} we make sure that the union of pairwise
intersections of any three crosses (possibly, coinciding) on $R(3\times 5)$
board contains at most 8 cells. The cases, in which this union contains 7 or 8
cells, are drawn in fig.~\ref{fig:crosses3x5},  these are the cases, when the
centers of two crosses belong to the same vertical or the same horizontal (including
the case, when both centers are in one cell). In all these cases, the union of
pairwise intersections of three crosses occupies one whole horizontal of the
board, and it occupies less than a half of cells in each of the two other horizontals. This means that the union of two such sets cannot cover the board
completely.

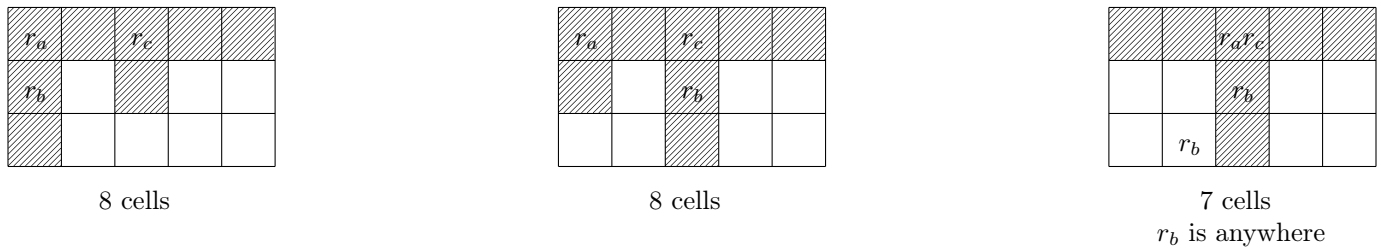
\begin{figure}[h]
\setlength{\unitlength}{20pt}
\footnotesize
\leavevmode
\begin{picture}(5,4)(0,-1)
\multiput(0,0)(0,1){4}{\line(1,0){5}}
\multiput(0,0)(1,0){6}{\line(0,1){3}}
\put(0.3,2.3){$r_a$}
\put(0.3,1.3){$r_b$}
\put(2.3,2.3){$r_c$}
\multiput(0,2)(1,0){5}{\shmtt Z}
\put(0,1){\shmtt Z}
\put(0,0){\shmtt Z}
\put(2,1){\shmtt Z}
\put(1.7,-.8){8 cells}
\end{picture}
\hfill
\begin{picture}(5,4)(0,-1)
\multiput(0,0)(0,1){4}{\line(1,0){5}}
\multiput(0,0)(1,0){6}{\line(0,1){3}}
\put(0.3,2.3){$r_a$}
\put(2.3,1.3){$r_b$}
\put(2.3,2.3){$r_c$}
\multiput(0,2)(1,0){5}{\shmtt Z}
\put(0,1){\shmtt Z}
\put(2,0){\shmtt Z}
\put(2,1){\shmtt Z}
\put(1.7,-.8){8 cells}
\end{picture}
\hfill
\begin{picture}(5,4)(0,-1)
\multiput(0,0)(0,1){4}{\line(1,0){5}}
\multiput(0,0)(1,0){6}{\line(0,1){3}}
\put(2,0){\shmtt Z}
\put(2,1){\shmtt Z}
\put(2.05,2.3){$r_a$}
\put(2.3,1.3){$r_b$}
\put(1.3,0.3){$r_b$}
\put(2.5,2.3){$r_c$}
\multiput(0,2)(1,0){5}{\shmtt Z}
\put(1.7,-.8){7 cells}
\put(.9,-1.4){$r_b$ is anywhere} 
\end{picture}
\caption{Union of pairwise intersections of three crosses on board
$R(3\times 5)$.}%
\label{fig:crosses3x5}%
\end{figure}

\smallbreak
{\bf Lose3)}
The argument below was proposed by Oleg Chemokos. 

We fix some strategies of the chess players $\mathcal L$ and $\mathcal R$ and verify that one can find positions for the kings such that both kings avoid a check.
In our standard notations each cell $i$ on $L(2\times 4)$ board determines three $i$-weak cells (see fig.~\ref{fig:2x4and3x3}{\it a}). This set of three weak cells can consist of any three cells in one row.

The strategy of chess player~$\mathcal L$ is given by labelling each cell on $R(3\times 4)$ board. Paint in white the cells of $R(3\times 4)$ board containing the labels corresponding to the first row of $L(2\times 4)$ board, paint in black the other cells. Without loss of generality we may assume that the number of white cells on the board is less than or equal the number of black cells. The following three cases cover all the possibilities, for which this inequality can be realized.

1. One of the horizontals of $R(3\times 4)$ board (for definiteness the first one) contains three white cells $u_1$, $u_2$, $u_3$ and one more horizontal (the second) contains two white cells $u_4$ and $u_5$. Then the first row of $L(2\times 4)$ board contains a cell~$\ell$ such that the label $\ell$ occurs in the first two horizontals of $R(3\times 4)$ board at most once, and, moreover, if so, the label $\ell$ occurs in the first horizontal, say, in cell $u_1$. The other cells of the first row on $L(2\times 4)$ board are $(\ell+4)$-weak, and cells $u_2$, $u_3$, $u_4$ and $u_5$ must belong to the same cross by lemma~\ref{lem:chess-strategy}, which is not true.

2. Each horizontal of $R(3\times 4)$ board contains two white cells. Then we choose a cell~$\ell$ in the first row of $L(2\times 4)$ board, such that the label $\ell$ occurs on $R(3\times 4)$ board at most once (for definiteness, in the third horizontal). The other cells in the first row of $L(2\times 4)$ board are $(\ell+4)$-weak, and the corresponding labels in the first two horizontals of $R(3\times 4)$ board are not covered by one cross.

3. One horizontal contains four white cells and two other horizontals contain one white cell each. Then we replace ``'black'' and ``white'' and consider the first case.

The obtained contradiction proves that the strategy is losing.

\smallbreak
{\bf Lose4)}
We number the cells of board $L$, as in fig.~\ref{fig:2x5and3x3}\,a). The strategy of chess player~$\mathcal L$ is given by writing the  number from 1 to 10 in each cell of $R(3\times 3)$ board  (these are the numbers of cells on $L(2\times 5)$ board). Since  $L(2\times 5)$ board has two rows only, there exist two horizontals on $R(3\times 3)$ board and two cells in each of them, such that the four labels in these cells correspond to the cells (some of them can coincide) belonging to the same row of $L(2\times 5)$ board. Let $j$-th cell in the other row be $i$-weak with respect to all these cells.

For example, let labels 1, 2, 3, 4 be located on $R(3\times 3)$ board, as in fig. \ref{fig:2x5and3x3}\,b). Then the number~10 is 1-, 2-, 3- and 4-weak simultaneously. This means that the rook in the cell~$r_{10}$ of $R(3\times 3)$ board attacks the cells with labels 1, 2, 3 and 4. But this is impossible:  to attack labels 1 and 2, it must be located in the upper row of $R(3\times 3)$ board, and to attack 3 and 4, it must be located in the bottom row.

By the same reason the general case is also impossible: the cell~$r_j$ must be located
in two horizontals of~$R(3\times 3)$ simultaneously.

\begin{figure}[h]
\setlength{\unitlength}{18pt}
\footnotesize
\begin{minipage}[t]{0.56\linewidth}
\leavevmode
\vbox{\hsize=3.3cm
\begin{tabular}{|c|c|c|c|c|}
\hline
1&2&3&4&5\\
\hline
6&7&8&9&10\\
\hline
\end{tabular}
\par
\vskip5pt
a) Cell numbers on board $L$
}
\hfill
\begin{picture}(5,4)(0,-1)
\multiput(0,0)(0,1){4}{\line(1,0){3}}
\multiput(0,0)(1,0){4}{\line(0,1){3}}
\put(0.3,2.3){1}\put(1.3,2.3){2}%
\put(0.3,0.3){3}\put(1.3,0.3){4}%
\put(-1.5,-1.){b) Strategy of player $\mathcal L$}
\end{picture}
\caption{Seek a strategy for game $L(2\times 5)$, $R(3\times 3)$.}%
\label{fig:2x5and3x3}%
\end{minipage}
\hfill
\begin{minipage}[t]{0.4\linewidth}
\begin{center}
\begin{picture}(4,3)%
\multiput(0,0)(0,1){4}{\line(1,0){4}}
\multiput(0,0)(1,0){5}{\line(0,1){3}}
\put(3.5,1.3){$r_2$}
\put(3.05,1.3){$r_1$}
\put(0.05,0.3){$r_7$}
\put(0.5,0.3){$r_8$}
\put(3.3,2.3){$a$}
\put(1.3,0.3){$\times$}
\put(2.3,0.3){$\times$}
\put(0.3,1.3){$b$}
\end{picture}
\caption{Strategy for case Lose5).}
\label{fig:33vs34}
\end{center}
\end{minipage}
\end{figure}

\smallbreak
{\bf Lose5)}
Assume that the chess players have a winning strategy. We number the cells of board~$L(3\times 3)$ by numbers from 1 to 9. Then the strategy of chess player~$\mathcal R$ is specified by a placement of nine symbols: $r_1$, $r_2$, \dots, $r_9$ on board~$R(3\times 4)$. And the strategy of chess player~$\mathcal L$ is specified by writing a number from 1 to 9 in each cell of $R(3\times 4)$ board.

Claim 1. If the cells $u$, $v$, $w$ belong to three different rows and three different columns of $L(3\times 3)$ board, then the labels  $r_u$, $r_v$ и $r_w$  belong to three different horizontals of $R(3\times 4)$ board.

Indeed, each cell of $L(3\times 3)$ board is either $u$-weak, or $v$-weak, or $w$-weak. By lemma~\ref{lem:chess-strategy} this implies that each label on $R(3\times 4)$ board belongs to $r_u$-, $r_v$- or $r_w$-cross. This is possible if the labels $r_u$, $r_v$ and $r_w$ are in different horizontals only.

Claim 2. There are two possible cases of the placement of symbols $r_1$, $r_2$, \dots, $r_9$ on $R(3\times 4)$ board:

1) either the symbols $r_1$, $r_2$,  $r_3$ are located in one horizontal of $R(3\times 4)$ board, the symbols $r_4$, $r_5$,  $r_6$ are located in another horizontal, and the symbols $r_7$, $r_8$,  $r_9$ are in the third one;

2) or the symbols $r_1$, $r_4$,  $r_7$ are located in one horizontal of $R(3\times 4)$ board, the symbols $r_2$, $r_5$,  $r_8$ are located in another horizontal, and the symbols $r_3$, $r_6$,  $r_9$ are in the third one.

The claim is proved by moderately nasty brute force with the help of claim 1.

Put rooks in all cells~$r_i$ of $R(3\times 4)$ board (we put in a cell as many rooks as there are symbols~$r_i$ in it). Each cell $i$ on $L(3\times 3)$ board determines four $i$-weak cells which are located in two rows and two columns.

Claim 3. Each cell on $R(3\times 4)$ board (let it contain a label $i$)
is under attack when the rook stands on cells labeled $r_j$, where $j$ is $i$-weak number. Two of this ``dangerous'' cells are located in the same horizontal, and the other two belong to another horizontal.

The claim follows from claim 2 . This means that we put several rooks on some cells.

Now we prove that no winning strategy with these properties exist. By claim 2, the first horizontal of $R(3\times 4)$ board contains at most three labels $r_i$. Therefore the first horizontal $R(3\times 4)$ board contains an ``empty'' cell, i.\,e.~the cell, containing no symbols~$r_i$, denote it by~$a$. For definiteness let it be in the fourth vertical (fig.~\ref{fig:33vs34}).
By claim 3, four rook's attacks are directed to this cell, and two of these four rooks are in one horizontal, and another two are in another horizontal. This means that two rooks are certainly located in one of the cells of the fourth vertical. For definiteness let this cell be located in the second horizontal.
By claim 2, the second horizontal contains three rooks in total, and we have established that two of them are in one cell.
Therefore, there are two ``empty'' cells in the second horizontal. Let us choose the one  above which there is no more than one rook stands in the first horizontal. Let this cell be in the first column, denote it~$b$. There are four rooks attacks from two pairs of rooks located in two rows directed to chosen cell. One pair of rooks is obviously located in the second horizontal, and another pair is located in the third horizontal (there is at most one rook above cell~$b$ in the first horizontal). Now we see that one of the cells in the third horizontal, in the second or in the third vertical, cannot gather four rook's attacks from two different horizontals, a contradiction.

\smallbreak
{\bf Lose6)}
Assume that the chess players have a winning strategy. The strategy of chess player~$\mathcal R$ is given in standard notations by placement of the four symbols $r_1$, $r_2$, $r_3$, $r_4$ on $R(5\times 5)$ board. There is at least one cell $Q$ on $R(5\times 5)$ board, not belonging to any of four crosses determined by these symbols. The strategy of chess player~$\mathcal L$ is specified by writing a number from 1 to 4 in each cell of $R(5\times 5)$ board. Without loss of generality cell~$Q$ is labeled by~1.
Let the referee put the kings on cell~$Q$ on $R(5\times 5)$ board and on cell~4 on $L(2\times 2)$ board. Then player~$\mathcal L$ puts his rook on cell~1 of $L(2\times 2)$ board, and player~$\mathcal R$ puts the rook on cell~$r_4$ of $R(5\times 5)$ board. None of the rooks attacks the king. The chess players lost.

The theorem is completely proved.
\end{proof}

\subsection{Queen check}

Consider a variation of the game, where the players put queens instead of rooks. Call this game \emph{Queen check}.

\begin{lemma}\label{lemQueen4545}
  The players win in Queen check game on boards  $L(4\times 5)$, $R(4\times 5)$.
\end{lemma}

\begin{proof}
  Paint the cells of both boards as shown in fig.~\ref{fig:queen45vs45}, a). Let
  both chess players put their queens only on the cells marked with queens, and let the first chess player under the assumption ``the Kings are on cells
  of the same color'', and the second under the assumption ``The
  kings are located on cells of different colors''.
\end{proof}

However we can use also usual chess coloring instead of ``exotic''  coloring
as above. Indeed, the queen on the cell~$c2$ holds under attack all the
cells of the same color in chessboard coloring! And the same is true for the cell $c$3,
fig.~\ref{fig:queen45vs45}, b).

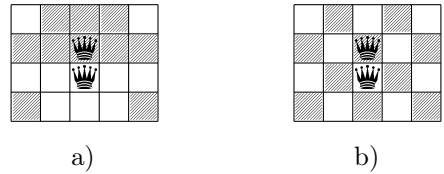
\begin{figure}[h]
\footnotesize
\hfil
\begin{minipage}[b]{0.5\linewidth}
\setlength{\unitlength}{11pt}
\begin{center}
\begin{picture}(5,4)%
\multiput(0,0)(0,1){5}{\line(1,0){5}}
\multiput(0,0)(1,0){6}{\line(0,1){4}}
\multiput(1.05,3.05)(1,0){3}{\shmt Z}
\multiput(0.05,2.05)(1,0){2}{\shmt Z}
\multiput(3.05,2.05)(1,0){2}{\shmt Z}
\put(0.05,0.05){\shmt Z}
\put(4.05,0.05){\shmt Z}
\put(2.05,2.05){\shmt l}
\put(2.05,1.05){\shmt q}
\put(2.05,-1.5){a)}
\end{picture}
\hfil
\begin{picture}(5,4)%
\multiput(0,0)(0,1){5}{\line(1,0){5}}
\multiput(0,0)(1,0){6}{\line(0,1){4}}
\multiput(0.05,0.05)(2,0){3}{\shmt Z}
\multiput(0.05,2.05)(4,0){2}{\shmt Z}
\multiput(1.05,1.05)(2,0){2}{\shmt Z}
\multiput(1.05,3.05)(2,0){2}{\shmt Z}
\put(2.05,2.05){\shmt l}
\put(2.05,1.05){\shmt q}
\put(2.05,-1.5){b)}
\end{picture}
\end{center}
\caption{Queen check on $4\times5$ boards.}%
\label{fig:queen45vs45}%
\end{minipage}

\end{figure}

The next statement has been found by computer, the proof was found by
N.~Kononenko.

\begin{lemma}
  The players win in ``Check by queen'' game on boards  $L(4\times 4)$,
  $R(5\times 5)$
\end{lemma}

\begin{proof}
  Specify strategy of the chess players. Label $R(5\times 5)$ board as in
  fig.~\ref{fig:queen4x4and5x5}\,a). Seeing the fellow's king on the cell with
  label~$j$, chess player~$\mathcal L$ puts his queen on $L(4\times 4)$ board in
  the cell, labeled by number~$j$, fig.~\ref{fig:queen4x4and5x5}\,b). So, chess
  player~$\mathcal L$ uses only four positions for his queen. The numbers in the
  cells of $L(4\times 4)$ board in fig.~\ref{fig:queen4x4and5x5}\,c) show, from
  which positions the queen of chess player~$\mathcal L$ does not attack this
  cell. For example, the numbers 1 and 2 in the lower left corner mean that the
  lower left corner cell of $L(4\times 4)$ board is not under attack by the
  queen located at 1-st and in 2-nd positions,  shown in
  fig.~\ref{fig:queen4x4and5x5}\,b), and ``--'' means that the cell is under
  attack from all positions.

\begin{figure}[h]
\setlength{\unitlength}{11pt}
\footnotesize
\begin{minipage}[b]{0.3\linewidth}
\begin{center}
\tabcolsep=3.6pt
\begin{tabular}{|c|c|c|c|c|}
\hline
3&1&3&1&3\\
\hline
1&3&3&3&4\\
\hline
3&3&3&3&3\\
\hline
2&3&3&3&4\\
\hline
3&2&3&2&3\\
\hline
\end{tabular}

\begin{tabular}{ccccc}
{\it a}&{\it b}&{\it c}&{\it d}&{\it e}\\
\end{tabular}

\par
\vskip5pt
a) Strategy of player $\mathcal L$
\end{center}
\end{minipage}
\hfil
\setlength{\unitlength}{15pt}
\begin{picture}(4,5)(0,-1)
\multiput(0,0)(0,1){5}{\line(1,0){4}}
\multiput(0,0)(1,0){5}{\line(0,1){4}}
\put(0.3,0.3){4}
\put(2.3,1.3){2}
\put(2.3,2.3){3}
\put(1.3,2.3){1}
\put(-1.7,-1.2){b) Where $\mathcal L$ puts the queen}%
\end{picture}
\hfil
\begin{picture}(4,4)(0,-1)
\multiput(0,0)(0,1){5}{\line(1,0){4}}
\multiput(0,0)(1,0){5}{\line(0,1){4}}
\put(0.3,3.3){3}\put(1.07,3.3){2,4}\put(2.3,3.3){4}\put(3.05,3.3){1,2}
\put(0.3,2.3){2}\put(1.3,2.3){4}\put(2.3,2.3){--}\put(3.3,2.3){4}
\put(0.3,1.3){3}\put(1.3,1.3){--}\put(2.3,1.3){4}\put(3.05,1.3){1,4}
\put(0.05,0.3){1,2}\put(1.3,0.3){3}\put(2.3,0.3){1}\put(3.3,0.3){3}
\put(-1.3,-1.2){c) Instruction for $\mathcal R$}
\end{picture}
\caption{Queen check on $L(4\times 4)$ and $R(5\times 5)$ boards.}%
\label{fig:queen4x4and5x5}%
\end{figure}

  Seeing the king on $L(4\times 4)$ board, chess player~$\mathcal R$ with help
  of fig.~\ref{fig:queen4x4and5x5}\,c) immediately understands, from which
  ``unfavorable'' positions the queen of his partner cannot put the king in
  check. Therefore he must put his queen on $R(5\times 5)$ board so that it
  attacks all the cells, sending the queen of chess player~$\mathcal L$ to a
  unfavorable position.

  For unfavorable positions 1, 2 it is possible to put the queen on cell~{\it
    b}3; \  for 1, 4 on cell~{\it c}4;  \ for 2, 4 on cell~{\it c}2; \ for 3 on
  cell~{\it c}3.
\end{proof}

We found by computer that in Queen check game the chess players lose on boards $L(3\times 4)$, $R(7\times 7)$ and $L(4\times 5)$, $R(5\times 5)$.

The following statement was suggested to us by S.~Berlov. It generalizes the argument of lemma~\ref{lemQueen4545}.

\begin{lemma}
  Consider a variation of Queen сheck game in which five chess players
  are located so that each of them sees the boards of the others but does not
  see his own board. All the boards have size $11\times 11$.  As in the initial
  game, the referee puts one king on each board, and the chess players
  simultaneously point to the cells on their own boards, where the queen has to be put.

  The chess players win in this game.
\end{lemma}

\begin{proof}
  On $11\times 11$ board, one can place 5 queens that keep all the cells under attack (for example, $b$4, $d$10, $f$6, $h$2 and~$j$8). During the game, the
  chess players will put their queens on the these 5 positions only.
  We number these positions from 0 to 4. In each cell of $11\times 11$
  board, we place the number of any of these queens, that holds
  this sell under attack. We assume that this labelling is applied all boards. When the referee puts the
  king on some cell of a board, the label of this cell is called the \emph{weight}
  of the king.

  The strategy of chess players is as follows: let the $k$-th player check the
  hypothesis that the sum of weights of all kings equals $k$ modulo~5. Each
  player sees all the kings except his own and calculates the weight of
  his king at which the hypothesis is correct. Then the player puts his queen on the
  position, which number equals the calculated weight.
\end{proof}

\subsection{Check with other chess pieces}

In the games Bishop check or Knight check the check declaration means
that the chess player guesses the color of the cell in which the king is standing. Therefore the chess
players can win in these games only on small boards, where all cells of each
color can be attacked from one point.

Consider the game King check (the referee puts on the board a ``good king'' and the chess
player puts on the board an ``evil king'' who must put the good one in check.

\begin{theorem}
  For the King check game on the boards $L(a\times b)$, $R(c\times d)$, denote by
  $\ell$ the number of elements in the maximal set of cells on $L(a\times b)$
  board, such that no two cells can be under attack of the same king. Define
  a number $r$ for  $R(c\times d)$ board analogously. Then the chess players win if and only if $\ell= r=2$ or one of the numbers $\ell$, $r$ equals~1.
\end{theorem}

\begin{proof}
  Choose sets $S_L$, $S_R$ of cells on $L(a\times b)$ and $R(c\times d)$ boards
  so that no two cells in these sets can be under attack of the same king,
  $|S_L|=\ell$, $|S_R|=r$. Let the referee make things easier for the chess players by promising that he will place
  the kings on the cells of the sets $S_L$ and $S_R$ only. Since the ``evil king''
  cannot attack two cells simultaneously, we may assume that the chess players
  just try to guess where the ``good king'' stands, or, which is the same, to
  guess hat colors the \hats game on the graph $P_2$ with hatnesses $\ell$, $r$, which 
  is possible only if $\ell= r=2$ or when one of the numbers $\ell$, $r$ equals~1.

  It remains to show that in these cases the chess players win. For
  $\ell=1$ or $r=1$ this is obvious. The maximal possible board for $\ell=r=2$ is
  $3\times 6$, because no two corner cells in  $4\times 4$ board as well as no two cells
  of $1\times 7$ board with coordinates 1, 4, 7 are attacked by the same king.
  On $3\times 6$ board, the chess players easily win by splitting te board into two
  halves of sizes $3\times 3$ and checking the hypotheses ``good kings are in the
  same/different halves''.
\end{proof}

\section{Analysis of \hats game on a cycle}

According to results of W.~Szczechla~\cite{cycle_hats}, the sages have some
difficulties in the game on cycle $C_n$ already in the case, when all hatnesses are
equal to 3. In that case, the winn of the sages is possible only if $n=4$ or $n$ is divisible by
3. If one of the sages on any cycle has hatness 4 (and all others have hatness 3), the sages lose ~\cite[corollary
8]{cycle_hats}.

The following theorem gives the list of games on cycles containing a vertex of hatness 2,
where the sages win.

\begin{theorem}\label{thm:cycle-win}
  Let $G$ be cycle $C_n$, and $h$ be the hat function such that $2
  \leq h(v) \leq 4$ for all vertices $v$. Let $A\in V(G)$,  $h(A) = 2$. Then the game $\HG=\langle {G, h} \rangle$ is winning in the following cases.
  \begin{enumerate}
    \item $n=3$;
    \item there is one more vertex with hatness $2$ other than $A$;
    \item both neighbors of vertex $A$ have hatness $3$;
    \item one neighbor of $A$ and the vertex following it are of hatness $3$.
  \end{enumerate}
\end{theorem}

\begin{proof}
  If $h_1(v)\leq h_2(v)$ for all $v\in V(G)$, then the winning in $\langle
  {G, h_2} \rangle$ implies the winning in $\langle {G, h_1} \rangle$, or, which is
  the same, the losing in $\langle {G, h_1} \rangle$ implies the losing in $\langle {G,
    h_2} \rangle$. This is obvious, because the winning strategy for $\langle {G, h_2}
  \rangle$ can be used as a winning strategy for $\langle {G, h_1} \rangle$, in
  which instead of ``non-existing'' colors the sages say any 
  ``existing''ccolor. Therefore, to prove the theorem, it suffices to check the winning 
  for the cases when the hat function is ``maximal'' (in the sense of definition in subsection~\ref{subsubsec:product}).

For each statement of the theorem, we give below the maximal hat functions and
the proofs that the sages win. We recall $h_4^{A2B2}$  is the hat
function, which values are equal to 4 in all the vertices other than $A$ and $B$,
where $h(A)=2$, $h(B)=2$.

\begin{enumerate}
\item $C_3$ with hatnesses 2, 4, 4. The sages win by
  corollary~\ref{cor:triangle244-win}.

\item Game $\langle {C_n, h_4^{A2B2}} \rangle$ is winning, because  it
  contains a path with hatnesses $2,4,\dots,4,2$, where the sages win by
  corollary~\ref{cor:path2442}.

\item The game $\langle {C_n, h_4^{A2B3C3}} \rangle$, where $B$ and $C$ are the
  neighbors of $A$ is winning by corollary~\ref{cor:cycle323}.
\item The game $\langle {C_n, h_4^{A2B3C3}} \rangle$, where $A$, $B$, $C$ are
  three consequent vertices, is winning by corollary~\ref{cor:cycle_233}.
\end{enumerate}
\end{proof}

\begin{conjecture}
  Let $G$ be a cycle $C_n$ and let $h$ be a hat function such that $2
  \leq h(v) \leq 4$ for every vertex $v$. Let $A\in V(G)$ be
  such that $h(A) = 2$. Then the game $\HG=\langle {G, h} \rangle$ is winning only
  in the cases listed in theorem~\ref{thm:cycle-win}.
\end{conjecture}

To prove the conjecture it suffices to prove that the following two games are losing.

\begin{enumerate}
\item $\langle {C_n, h_3^{A2B4C4}} \rangle$ $(n \geq 4)$, where sages $B$
  and $C$ are the neighbors of sage $A$. The loss in this game for $n=4$ is
  proved in theorem~\ref{thm:Rook_chess},~\ref{itm:lose2344} in the language
  of Rook check game. For $n\leq 7$ the loss was checked on computer by
  reduction to  SAT~\cite{Kokhas2018}. This allows us to assume
  that for $n\geq 8$ the game is losing too, but we have no proof of this fact.

\item $\langle {G_n, h_3^{A2B4C3D4}} \rangle$ $(n \geq 4)$, where the sages $B$
  and $C$ are the neighbors of sage $A$, and sage $D\ne A$ is the second neighbor of
  sage $C$. The loss of this game for $n=4$ is proved in
  theorem~\ref{thm:Rook_chess},~\ref{itm:lose2434}. For $n\leq 7$ the loss was
  checked by computer. This allows us to assume that for $n\geq 8$ the game
  is losing, but we still have no proof of this fact too.
\end{enumerate}

\section{Conclusion}

In the present paper we certainly prove that the variation of \hats game in
question is a real gem of combinatorics.  The firework of ideas that arise when considering
different approaches to the game is mesmeriring and awakens the imagination. In the
same time, the computational complexity of the game prevents from putting forward hasty
conjectures and effectively protects the game from a complete analysis.

\printbibliography

\end{document}